\documentclass[11pt,amsfonts, epsfig]{amsart}
\usepackage{amsmath, amscd, amssymb}
\usepackage{graphpap, color}
\usepackage{mathrsfs}
\usepackage{pstricks}
\usepackage{color}
\usepackage{cancel}
\usepackage[mathscr]{eucal}
\usepackage{pstricks}
\usepackage{color}
\usepackage{cancel}
\usepackage{verbatim}
\usepackage{latexsym}
\usepackage[all]{xy}

\usepackage{lipsum}
\makeatletter
\g@addto@macro{\endabstract}{\@setabstract}
\newcommand{\authorfootnotes}{\renewcommand\thefootnote{\@fnsymbol\c@footnote}}%
\makeatother

\usepackage{upgreek}

 \usepackage{wasysym}
\usepackage{ esint }

\def\vdim{\mathrm{vir}.\dim}

\def\virt{^{\vir}}
\def\virtloc{\virt_\loc}

\def\Tdec{^{T,\mathrm{dec}}}

\numberwithin{equation}{section}

\def\gm{\GG_m}

\def\sO{{\mathscr O}}
\def\sC{{\mathscr C}}

\def\sN{{\mathscr N}}
\def\sL{{\mathscr L}}

\def\sO{\mathscr{O}}

\def\sE{\mathscr{E}}

\def\sA{\mathscr{A}}

\def\sR{\mathscr{R}}

\newcommand{\CC}{\mathbb{C}}

\newcommand{\EE}{\mathbb{E}}

\newcommand{\PP}{\mathbb{P}}
\newcommand{\QQ}{\mathbb{Q}}

\newcommand{\ZZ}{\mathbb{Z}}
\newcommand{\GG}{\mathbb{G}}

\newcommand{\TT}{\mathbb{T}}

\def\sK{{\mathscr K}}

\newcommand{\bL}{\mathbf{L}}

\def\redd{{\mathrm{red}}}

\def\Gm{{\bG_m}}
\def\upmo{^{-1}}

\newcommand{\vir}{ {\mathrm{vir}} }
\newcommand{\ev}{ \mathrm{ev} }
\newcommand{\vg}{\vec{g}}

\newcommand{\vS}{\vec{S}}

\def\Wfix{\cW_\Ga} 
\def\fixW{\cW^{\,\Gm}}

\newcommand{\cal}{\mathcal}

\def\cC{{\cal C}}
\def\cD{{\cal D}}

\def\cF{{\cal F}}

\def\cL{{\cal L}}
\def\cM{{\cal M}}
\def\cN{{\cal N}}
\def\cO{{\cal O}}
\def\cP{{\cal P}}

\def\cR{{\cal R}}

\def\cU{{\cal U}}
\def\cV{{\cal V}}
\def\cW{{\cal W}}

\def\cS{{\cal S}}
\def\cX{X} 
\def\cY{{\cal Y}}

\def\cZ{{\cal Z}}

\def\fB{\mathfrak{B}}
\def\fC{\mathfrak{C}}

\def\ft{\mathfrak{t}}

\def\v1{{\vec{1}}}

\def\sP{{\mathscr P}}

\def\fl{{\mathrm{fl}}}

\def\mapright#1{\,\smash{\mathop{\lra}\limits^{#1}}\,}
\def\toright#1{\,\smash{\mathop{\to}\limits^{#1}}\,}

\def\cDg{\cD_{\ga}}

\newcommand{\vd}{\vec{d}}

\def\dual{^{\vee}}

\def\sta{^\ast}
\def\st{^{\mathrm{st}}}
\def\virt{^{\mathrm{vir}}}
\def\upmo{^{-1}}
\def\sta{^{\ast}}

\def\sta{^*}

\def\lra{\longrightarrow}

\def\lsta{_{\ast}}

\def\cDggg{\cD_{\ga,\nu}}

\newcommand{\Del}{\Delta}
\newcommand{\Si}{\Sigma}
\newcommand{\Ga}{\Gamma}

\newcommand{\ep}{\epsilon}
\newcommand{\lam}{\lambda}
\newcommand{\si}{\sigma}

\def\lrga{_{(\ga)}}

\def\begeq{\begin{equation}}
\def\endeq{\end{equation}}
\def\and{\quad{\rm and}\quad}
\def\bl{\bigl(}
\def\br{\bigr)}
\def\defeq{:=}

\def\sub{\subset}
\def\Ao{{\mathbb A}^{\!1}}

\def\Po{{\mathbb P^1}}
\def\and{\quad\text{and}\quad}

\DeclareMathOperator{\pr}{pr}

 \DeclareMathOperator{\Aut}{Aut}

\DeclareMathOperator{\spec}{Spec}

\def\lggd{}

\def\cWgg{\cW\lggd}

\newtheorem{prop}{Proposition}[section]

\newtheorem{theo}[prop]{Theorem}
\newtheorem{lemm}[prop]{Lemma}
\newtheorem{coro}[prop]{Corollary}
\newtheorem{rema}[prop]{Remark}

\newtheorem{defi}[prop]{Definition}
\newtheorem{definition}[prop]{Definition}

\newtheorem{defi-prop}[prop]{Definition-Proposition}
\newtheorem{defi-theo}[prop]{Definition-Theorem}

\def\Ob{\cO b}

\def\lloc{_{\mathrm{loc}}}

\def\loc{\mathrm{loc}}
\def\lloc{_\loc}

\def\ev{\text{ev}}

\def\sta{^\ast}

\def\sO{{\mathscr O}}

\def\sR{{\mathscr R}}

\def\beq{\begin{equation}}
\def\eeq{\end{equation}}

\def\Pf{{\PP^4}}

\def\bee{\begin{equation}}
\def\eeq{\end{equation}}

\def\sC{{\mathscr C}}

\def\fA{{\mathfrak A}}

\def\bd{{\mathbf d}}

\let\eps=\epsilon

\def\ti{\tilde}

\def\mapright#1{\,\smash{\mathop{\lra}\limits^{#1}}\,}

\def\un{^{\mathrm{un}}}

\def\broad{{\mathrm{br}}}
\def\narrow{{\mathrm{na}}}

\def\lred{_{\mathrm{red}}}

\def\bmu{{\boldsymbol \mu}}

\def\Gm{T}

\def\cDgg{{\cD\lnuga}}
\def\cDggg{\cD_{\ga,[\nu]}}
\def\cWgg{\cW_\Ga}
\def\cYgg{\cY_{\ga,\nu,e}}
\def\cWge{\cW_{\Ga'}}
\def\cWee{\cW_{\bar e}}

\let\ga=\Ga

\def\lorho{_{^{(1,\rho)}}}
\def\lophi{_{^{(1,\varphi)}}}
\def\lnuga{_{\ga,\nu}}
\def\bAB{\fA_{\Ga,e|\fB}} 

\begin{document}

\title[A vanishing]{A vanishing associated with irregular \\ MSP fields}

\author{Huai-Liang Chang}
\author{Jun Li}
\address{Mathematics Department, Hong Kong University of Science and Technology}
\address{Mathematics Department, Stanford
University and Shanghai Center for Mathematical Science, Fudan University}
\thanks{H.L. Chang was partially supported by HK GRF grant 16301515 and 16301717; J. Li was partially supported by NSF
grant DMS-1159156, DMS-1564500, and DMS-1601211.}

\maketitle

\begin{abstract}In \cite{CLLL} and \cite{CLLL2}, the notion of Mixed-Spin-P field is introduced and their
moduli space $\cW_{g,\gamma,\bd}$ together with a $\CC\sta$ action is constructed.
Applying virtual localization to their virtual classes
$[\cW_{g,\gamma,\bd}]\virt$, polynomial relations among GW and FJRW invariants of Fermat quintics are
derived.

In this paper, we prove a vanishing of a class of terms in $[(\cal{W}_{g,\gamma,\bf{d}})^{\mathbb{C}^*}]\virt$.
This vanishing proves that in Witten's GLSM for Fermat quintics, the FJRW invariants (for all genus)
with insertions $2/5$ will determine the GW invariants of quintic Calabi-Yau through CY-LG phase transitions.

%
%
\end{abstract}

\section{Introduction}

In \cite{CLLL}, the authors introduced the notion of Mixed-Spin-P fields (abbre. MSP fields),
and constructed the properly supported $\GG_m$-equivariant virtual cycles of the
moduli spaces of these fields.
Applying virtual localization \cite{GP}, we obtained
relations among the GW invariants of quintic CY threefolds, and a class of FJRW invariants of the Fermat quintic.
Among the class of FJRW invariants involved there is a subclass of {\sl broad-like} FJRW invariants;
provided that this subclass all vanish, we obtain
polynomial relations among the GW invariants of quintic CY threefolds, and
FJRW invariants of the Fermat quintic with insertion $-\frac{2}{5}$.
This paper is devote to prove such a vanishing (Theorem \ref{main}).

\smallskip

Recall that an MSP field is a collection
\beq\label{MSP0}
\xi=( {\Si^\sC}, \sC, \sL, \sN,\varphi,\rho, \nu_1,\nu_2),
\eeq
consisting of a pointed twisted curve $\Sigma^\sC\sub\sC$, invertible sheaves
$\sL$ and $\sN$, and a collection of fields $(\varphi,\rho,\nu_1,\nu_2)$ (cf. Definition \ref{def-curve}).
The MSP field $\xi$ comes with numerical invariants: the genus $g$ of $\sC$, the
monodromy $\gamma_i$ of $\sL$ at the $i$-th marking $\Si^\sC_i$,
and the bi-degrees $d_0=\deg \sL\otimes\sN$ and $d_\infty=\deg \sN$.

Given $g$, $\gamma=(\gamma_1,\cdots,\gamma_\ell)$ and $\bd=(d_0,d_\infty)$, we let
$\cW\lggd$ be the moduli of stable MSP fields of numerical data $(g,\gamma, \bd)$.
It is a separated DM stack, locally of finite type. (The data $(g,\gamma,\bd)$ will be fixed throughout this
paper.)

As shown in \cite{CLLL, CLLL2}, the moduli $\cW\lggd$ is a $T=\GG_m$ DM stack (cf. \eqref{inv});
admits a $T$-equivariant perfect obstruction theory
and an invariant cosection
$\sigma_\cW: \Ob_{\cW\lggd}\to\sO_{\cW\lggd}$, giving rise to a cosection localized virtual cycle
\cite{KL} 
$$[\cW\lggd]\virtloc\in A^{\Gm}\lsta \cW\lggd^-,
$$
where $\cW\lggd^-$ is the vanishing locus of $\sigma$.
In \cite{CLLL}), it is proved that $\cW\lggd^-$ is proper and of finite type.

Following \cite{CLLL2}, we decompose the fixed locus $\cW\lggd^T$ into
disjoint open and closed substacks
$$\cW\lggd^T=\coprod_{\Ga\in\Delta^\fl}\cW_{(\Ga)},
$$
indexed by a set of (flat) decorated graphs $\Delta^\fl$. By the virtual localization
\cite{GP, CKL}, after inverting the generator $\ft\in A^1_T(pt)$,
\beq\label{van-1}
[\cW\lggd]\virtloc=
\sum_\Gamma \frac{[\cW_{(\Ga)}]\virtloc}
{e(N_{\cW_{(\Ga)}/\cW\lggd})}
\in \bl A^{\Gm}\lsta \cW\lggd^-\br_\ft.
\eeq

We call a graph a pure loop if it has no legs, has no stable vertices, and every vertex has
exactly two edges attached to it.
In \cite{CLLL2}, we divided the set $\Delta^\fl$ into regular and irregular graphs
(Definition \ref{regu}).

\begin{defi}
Let $Z\sub \cW^T$ be a proper closed substack, viewed as a $T$-stack with trivial $T$ action.
We say $\alpha\in A\lsta^T Z$ is weakly trivial, denoted by
$\alpha\sim 0$, if there is a closed proper substack $Z'\sub \cW^T$ with $Z\sub Z'$ so that $\alpha$ is mapped
to zero under the induced homomorphism $A\lsta^T Z\to A\lsta^T Z'$.
\end{defi}

In this paper, we will prove

\begin{theo}\label{main}
Let $\Ga$ be an irregular graph and not a pure loop, then $[\cW_{(\Ga)}]\virtloc\sim 0$.
\end{theo}

 Let $[\cdot]_0: A^\Gm\lsta \cW\lggd^-\to A_0(pt)$ be the proper pushforward induced by $\cW\lggd\to pt$.
Then Theorem \ref{main} implies that for the $\ga$ as stated in Theorem \ref{main}, and for any $\beta\in A\sta_T
\cW\lggd$,
$$\left[\beta\cap \frac{[\cW_{(\Ga)}]\virtloc}
{e(N_{\cW_{(\Ga)}/\cW\lggd})}\right]_0=0. 
$$

This vanishing theorem implies the only quintic FJRW invariants that contribute to the relations derived from the theory of
MSP fields are those with pure insertions $2/5$ (see \cite{CLLL2}).


\section{Irregular graphs}
\def\st{{\mathrm{st}}}
\def\wt{\mathrm{wt}}

In this section, we recall the notion of MSP fields, and decorated graphs associated
to $T$-invariant MSP fields. These notions and the proofs of the stated properties are taken from
\cite{CLLL2}.

\subsection{MSP fields}
Let $\bmu_5=\langle \zeta_5\rangle \le\GG_m$ be the subgroup of fifth-roots of unity, generated by $\zeta_5=\exp(\frac{2\pi\sqrt{-1}}{5})$. Let
$$\bmu_5^\narrow=\{(1,\rho), (1,\varphi),\zeta_5,\cdots,\zeta_5^4\}
\and \bmu_5^\broad=\{(1,\rho),(1,\varphi)\}\cup\bmu_5.
$$
Here $(1,\varphi)$ and $(1,\rho)$ are symbols, function as the identity element with special property; thus
the subgroup they generate
$\langle (1,\rho)\rangle=\langle (1,\varphi)\rangle=\{1\}\le \gm$ are the trivial subgroup.
Note that $\bmu_5^\narrow$ is by removing $1$ from $\bmu_5^\broad$. The data in $\bmu_5^\broad$ are called ``broad",
while that in $\bmu_5^\narrow$ are called ``narrow".

Let 
$$g\ge 0,\quad \gamma=(\gamma_1,\cdots,\gamma_\ell)\in (\bmu_5^\broad)^{\times\ell},\quad \bd=(d_0, d_\infty)\in
\QQ^{\times 2}.
$$
For an $\ell$-pointed twisted curve $\Si^\sC\sub \sC$, and for
$\alpha\in \bmu_5^\broad$, we agree 
$$\omega^{\log}_{\sC/S}=\omega_{\sC/S}(\Si^\sC),\and \Si^\sC_\alpha=\coprod_{\gamma_i=\alpha}\Si^\sC_i.
$$

\begin{definition}[{\cite{CLLL}}]
\label{def-curve}
A $(g,\gamma,\bd)$ MSP field $\xi$ is a collection \eqref{MSP0}
such that

\begin{enumerate}
\item[(1)] $\cup_{i=1}^\ell\Si_i^\sC = \Si^\sC\subset \sC$ is an $\ell$-pointed, genus $g$,
twisted curve such that the $i$-th marking $\Si^\sC_i$ is banded by the group $\langle\gamma_i\rangle\le \gm$;
\item[(2)] $\sL$ and $\sN$ are invertible sheaves on $\sC$, $\sL\oplus \sN$ representable,
$\deg \sL\otimes \sN=d_0$,   $\deg \sN=d_\infty$, and
the monodromy of $\sL$ along $\Si^\sC_i$ is
$\gamma_i$ when $\langle \gamma_i\rangle\ne\langle 1\rangle$;
\item[(3)] $\nu=(\nu_1, \nu_2)\in H^0( \sL\otimes\sN)\oplus  H^0( \sN)$, and $(\nu_1,\nu_2)$ is nowhere zero;
\item[(4)]
$\varphi=(\varphi_1,\ldots, \varphi_{5}) \in H^0(\sL)^{\oplus 5}$,  $(\varphi,\nu_1)$ is nowhere zero, and $\varphi|_{\Si^\sC_{(1,\varphi)}}=0$;
\item[(5)] $\rho \in H^0(\sL^{\vee\otimes 5}\otimes \omega^{\log}_{\sC/S})$,
$(\rho,\nu_2)$ is nowhere zero, and $\rho|_{\Si^\sC_{(1,\rho)}}=0$.
\end{enumerate}
We call $\xi$ (or $\gamma$) narrow if $\gamma\in (\bmu_5^\narrow)^\ell$. 
We call $\xi$ stable if $|\Aut(\xi)|<\infty$.
\end{definition}

The definition of monodromy can be found, say, in \cite{FJR, CLL}.
A typical example of monodromy is as follows. 
Consider $\sC=[\Ao/\bmu_5]$, where $\bmu_5$ acts on $\Ao=\spec \CC[x]$ via
$\zeta_5\cdot x=\zeta_5\upmo x$.
Then the $\sO_\sC$-module $x^{-2}\CC[x]$ has monodromy $\zeta_5^2$ at the stacky point.

\medskip

Throughout this paper, unless otherwise mentioned,  by
an MSP field $\xi$ we mean $\xi=(\Si^\sC,\sC,\sL,\cdots)$ as given in \eqref{MSP0}
with {\it{narrow}} $\gamma$.

\medskip
By the main theorem of \cite{CLLL}, the category $\cW$
of families of MSP-fields of data $(g,\gamma,\bd)$ is a separated DM stack.
The group $T=\GG_m$ acts on $\cW$ via
\beq\label{action}t\cdot (\sC,\Si^\sC, \sL,\sN,\varphi,\rho, \nu_1,\nu_2)=(\sC,\Si^\sC, \sL,\sN,\varphi,\rho, t\nu_1,\nu_2).
\eeq

%
The structure of $T$-invariant MSP fields can be summarized as follows.
Let $\xi\in \fixW$. Then
there is a homomorphism $h$ and $\Gm$-linearizations $(\tau_t,\tau_t')$ as shown
\beq\label{h-tau}
h: \Gm\lra \text{Aut}(\sC,\Si^\sC),\quad
\tau_t: h_{t\ast}\sL\lra \sL\and \tau'_t: h_{t\ast}\sN\lra \sN
\eeq
such that 
\beq\label{inv}
t\cdot (\varphi,\rho, \nu_1,\nu_2)=
(\tau_t,\tau_t')(h_{t\ast} \varphi,h_{t\ast}\rho,t\cdot h_{t\ast}\nu_1,h_{t\ast}\nu_2),\quad
t\in\Gm.
\eeq
(Here we allows fractional weight $T$ actions on curves, etc..)
We call such $\Gm$-actions and linearizations induced from $\xi\in \fixW$.
Since $\xi\in \cW^T$ is stable, such $(h,\tau_t,\tau_t')$ is unique.

Let $\bL_k$ be the one-dimensional weight $k$ $T$-representation. Let
\beq\label{LP}\sL^{\log}=\sL(-\Si^\sC\lophi)\and \sP^{\log}=\sL^{-5}\otimes\omega_{\sC}^{\log}(-\Si^\sC\lorho).
\eeq
 Then \eqref{inv} can be rephrased as
\beq\label{inv2}
(\varphi,\rho, \nu_1,\nu_2)\in H^0((\sL^{\log})^{\oplus 5}\oplus \sP^{\log}\oplus\sL\otimes\sN\otimes\bL_1\oplus\sN)^T.
\eeq

\subsection{Decorated graphs of $T$-MSP fields}
We describe the structure of $\fixW$, following \cite{CLLL2}. Let $\xi\in \fixW$, with $\sC$ its domain curve, etc.,
as in \eqref{MSP0}.
We decompose $\sC$ as follows: We let
$$\sC_0= (\nu_1=0)_\redd,\quad \sC_\infty=(\nu_2=0)_\redd,\quad \sC_1=(\rho=\varphi=0)_\redd\sub\sC;
$$
we let $\sA$ be the set of irreducible components of $\overline{\sC-\sC_0\cup\sC_1\cup\sC_\infty}$.
We let
$$\sC_{01}=\bigcup_{\sC_a\in\sA, \ \rho|_{\sC_a}=0}\sC_a,\quad
\sC_{1\infty}=\bigcup_{\sC_a\in\sA, \ \varphi|_{\sC_a}=0},\quad
\sC_{0\infty}=\bigcup_{\sC_a\in\sA, \ \rho|_{\sC_a}\ne0,\ \varphi|_{\sC_a}\ne 0}\sC_a. 
$$

We know that $\sC_0$, $\sC_1$ and $\sC_\infty$ are mutually disjoint,
and the action $h:\Gm\to \Aut(\sC,\Si^\sC)$ acts trivially on $\sC_0$, $\sC_1$ and $\sC_\infty$.
We also know that
every irreducible component $\sC_a\sub\sC_{01}$ (resp. $\sC_a\sub \sC_{1\infty}$; resp. $\sC_a\sub\sC_{0\infty}$)
is a smooth rational twisted curve with two $\Gm$-fixed points 
lying on $\sC_0$ and $\sC_1$ (resp. $\sC_1$ and $\sC_\infty$; resp. $\sC_0$ and $\sC_\infty$).


\medskip

We associate a decorated graph to each $\xi\in \fixW$.
For a graph $\Gamma$, besides its vertices $V(\Ga)$, edges
$E(\Ga)$ and legs $L(\Ga)$, the set of its flags is
$$
F(\Ga)=\{(e,v)\in E(\Ga) \times V(\Ga): v\in e\}.
$$
Given $\xi\in\cW^\Gm$, 
 let $\pi:\sC^{\text{nor}}\to\sC$ be its normalization. 
For any $y\in \pi\upmo(\sC_{\text{sing}})$, we denote by $\gamma_y$ the monodromy of $\pi\sta\sL$ along $y$.

\begin{defi}\label{graph1}
To $\xi\in \cW^\Gm$ we associate  a graph $\Ga_\xi$ as follows:
\begin{enumerate}
\item (vertex) let $V_0(\Ga_\xi)$, $V_1(\Ga_\xi)$, and $V_\infty(\Ga_\xi)$ be
the set of connected components of $\sC_0$, $\sC_1$, $\sC_\infty$ respectively, and
let $V(\Ga_\xi)$ be their union; 
\item (edge)   let $E_0(\Ga_\xi)$, $E_\infty(\Ga_\xi)$ and $E_{0\infty}(\Ga_\xi)$
be the set of irreducible components
of $\sC_{01}$, $\sC_{1\infty}$ and $\sC_{0\infty}$ respectively, and  let $E(\Ga_\xi)$
be their union; 
\item (leg) let $L(\Ga_\xi)\cong \{1,\cdots,\ell\}$ be the ordered set of markings of $\Si^\sC$,
$i\in L(\Ga_\xi)$ is attached to
$v\in V(\Ga_\xi)$ if $\Si_i^\sC\in \sC_v$;
\item (flag) $(e,v)\in F(\Ga_\xi)$ if and only if $\sC_e\cap \sC_v\ne \emptyset$. 
\end{enumerate}
We call $v\in V(\Ga_\xi)$ stable if $\sC_v\sub\sC$ is 1-dimensional, otherwise  unstable.
\end{defi}


We specify the decorations now. In the following,  let $V^S(\Ga_\xi)\subset V(\Ga_\xi)$ be the set of stable vertices.
Given $v\in V(\Ga_\xi)$,  let
$$S_v=\{\Si^\sC_j\in\sC_v\mid \Si^\sC_j\in\Si^\sC\},\quad E_v=\{e\in E(\Ga_\xi): (e,v)\in F(\Ga_\xi)\},
$$
consisting of the markings on $\sC_v$, and
of the edges attached to $v$, respectively.

For $v\in V^S(\Ga_\xi)$, we define
\beq\label{leg}\Si^{\sC_v}_{\text{inn}}=\Si^\sC\cap \sC_v,\quad
\Si^{\sC_v}_{\text{out}}=\overline{(\sC-\sC_v)}\cap\sC_v,\and \Si^{\sC_v}=\Si^{\sC_v}_{\text{inn}}\cup \Si^{\sC_v}_{\text{out}},
\eeq
called the inner, the outer, and the total
markings of $\sC_v$, respectively. They are respectively indexed by $S_v$, $E_v$ and $S_v\cup E_v$.

We adopt the following convention: for $a\in V(\ga_\xi)\cup E(\ga_\xi)$,
we define
$$d_{0a}= \deg \sL\otimes\sN|_{\sC_a},\quad d_{\infty a}= \deg\sN|_{\sC_a},
\and d_a=\deg \sL|_{\sC_a}=d_{0a}-d_{\infty a}.
$$
(This is consistent with $d_0=\deg \sL\otimes\sN$ and $d_\infty=\deg\sN$.)
For $e\in E_v$, we assign $\gamma_{(e,v)}$ according to the following rule:
\begin{enumerate}
\item when $d_e\not\in\ZZ$, assign $\gamma_{(e,v)}=e^{-2\pi\sqrt{-1}d_e}$; 
\item when $d_e\in\ZZ$ and $v\in V_\infty(\Ga_\xi)\cup V_1(\ga_\xi)$, 
assign $\gamma_{(e,v)}=(1,\varphi)$;
\item when $d_e\in\ZZ$ and $v\in V_0(\Ga_\xi)$, assign $\gamma_{(e,v)}=(1,\rho)$. 
\end{enumerate}


\begin{defi}\label{graph2}
We endow the graph $\Ga_\xi$ 
the following decoration:

\begin{itemize}
\item[(a)] (genus) Define $\vg: V(\Ga_\xi)\to \ZZ_{\geq 0}$ via $\vg(v)=h^1(\sO_{\sC_v})$.

\item[(b)] (degree) Define $\vd: E(\Ga_\xi)\cup V(\Ga_\xi)\to \QQ^{\oplus 2}$ via $\vd(a)=(d_{0a},d_{\infty a})$.

\item[(c)] (marking) Define $\vS: V(\Ga_\xi)\to 2^{L(\Ga_\xi)}$ via
$v\mapsto S_v\subset L(\Ga_\xi)$. 

\item[(d)] (monodromy) Define $\vec{\gamma}: L(\Ga_\xi)\to\bmu_5^\narrow$ via $\vec{\gamma}(\Si_i^\sC)=\gamma_i$.
\item[(e)] (level) Define $lev: V(\Ga_\xi)\to\{0,1,\infty\}$ by $lev(v)=a$ for $v\in V_a(\Ga_\xi)$.

\end{itemize}
\end{defi}

We form
\beq\label{VV0}
V^{a,b} (\Ga_\xi)=\{v\in V(\Ga_\xi)-V^S(\Ga_\xi): |S_v|=a, \,|E_v|=b\},
\eeq
and adopt the convention $V_j^S(\Ga_\xi)=V_j(\Ga_\xi)\cap V^S(\Ga_\xi)$;
same for $V_j^{a,b}(\Ga_\xi)$. 

We say $\Ga_\xi\sim \Ga_{\xi'}$ if there is an isomorphism of graphs $\Ga_\xi$ and $\Ga_{\xi'}$ that preserves
the decorations (a)-(e). We define
$$\Delta=\{\Ga_\xi\mid \xi\in\cW^T\}/\sim.
$$

\subsection{Decomposition along nodes}

We describe the decomposition of a $T$-MSP field along its $\Gm$-unbalanced nodes.

\begin{defi}
Let $\sC$ be a $\Gm$-twisted curve (i.e. twisted curve with a $T$-action)
and $q$ be a node of $\sC$. Let $\hat \sC_1$ and $\hat \sC_2$ be the
two branches of
the formal completion of $\sC$ along $q$. We call $q$ $\Gm$-balanced 
if $T_q\hat \sC_1\otimes T_q\hat \sC_2\cong \bL_0$ as $\Gm$-representations.
\end{defi}

For $\Ga\in\Delta$, we let
\beq\label{NGa}
N(\Ga)=V^{0,2}(\Gamma)\cup\{(e,v)\in F(\Ga)\mid v\in V^S(\Ga)\}.
\eeq
(Recall $v\in V^{0,2}(\Gamma)$ when $v$ associates to a node in $\sC$.)
Note that 
every $a\in N(\Ga_\xi)$ has its associated node $q_a$ of $\sC$.

\begin{defi}\label{def5.12}
We call $a\in N(\Ga_\xi)$ $\Gm$-balanced if 
the associated node $q_a$ is a $\Gm$-balanced node in $\sC$.
Let $N(\ga_\xi)\un\sub N(\Ga_\xi)$ be the subset of $T$-unbalanced.
\end{defi}


Clearly, if $v\in N(\Ga_\xi)$ is $T$-balanced, then $v\in V^{0,2}_1(\Ga_\xi)$.
Recall $d_e=\deg\sL|_{\sC_e}$.

\begin{lemm}[{\cite[Lemm.\,2.14]{CLLL2}}]\label{unstable-q}
For $v\in V_1^{0,2}(\Ga_\xi)$ with (distinct) $(e,v)$ and $(e',v)\in F(\Ga_\xi)$,
and letting $q_v=\sC_e\cap \sC_{e'}$ be the associated node,
then $q_v$ is $\Gm$-balanced
if and only if $d_{e}+d_{e'}=0$, and
$(\sC_e\cup\sC_{e'})\cap\sC_\infty$ is a node or a marking of $\sC$.
\end{lemm}



We comment that although  a $\Gm$-balanced $a\in N(\Ga_\xi)$ is characterized by $q_a$ being $\Gm$-balanced,
the previous reasoning shows that it can be 
characterized by the information of the graph $\Ga_\xi$.
Thus for any $\ga\in\Del$, we can talk about $N(\ga)\un\sub N(\ga)$ without referencing to any $\xi$.

\medskip

We now introduce flat graphs and regular graphs. 
We call a graph $\Ga\in\Delta$ flat if $N(\Ga)\un=N(\Ga)$.
We let $\Del^\fl\sub\Del$ be the set of flat graphs.
In case $N(\Ga)\un \subsetneq N(\Ga)$, we will associate a unique flat $\ga^\fl$,
called the flattening of $\ga$, as follows.
For each $\Gm$-balanced $v\in N(\Ga)$, which lies in $V_1^{0,2}(\Ga)$,
we eliminate the vertex $v$ from $\Ga$,  replace the two edges $e\in E_\infty(\Ga)$ and $e'\in E_0(\Ga)$
incident to $v$ by
a single edge $\ti e$ incident to the other two  vertices that are incident to $e$ or $e'$,
and demand that $\ti e$ lies in $E_{0\infty}$. For the decorations, we agree $\vec g(\ti e)=0$ and
$(d_{0\ti e}, d_{\infty\ti e})=(d_{\infty e}, d_{\infty e})$ (since $d_{0e'}=d_{\infty e}$, using
$d_{0e}=d_{\infty e'}=0$),
while keeping the remainder unchanged.
Let $\Ga^{\mathrm{fl}}$ be the resulting decorated graph after applying this procedure to all $\Gm$-balanced $v$ in
$N(\ga)$. 
We call $\Ga^{\mathrm{fl}}$, which is flat, the flattening of $\Ga$. 
We introduce
$$\Del^\fl=\{\Ga^\fl\mid \Ga\in\Delta\}/\sim.
$$
Indeed, it is easy to check that $\Del^\fl=\{\Ga\in\Delta\mid \text{$\Ga$ is flat}\}$.

\medskip
Given a flat $\ga\in\Del^\fl$, we define a $\ga$-framed $T$-MSP field to be a pair $(\xi,\eps)$, where $\eps:\ga_\xi^\fl\cong\Ga$
is an isomorphism (of decorated graphs). 
Like in \cite{CLLL2},
we can make sense of families of $\Ga$-framed $T$-MSP fields (cf. \cite[Section 2.4]{CLLL2}).
We then form the groupoid $\Wfix$ of $\Ga$-framed $T$-MSP fields with obviously defined arrows;
$\Wfix$ is a DM stack, with a forgetful morphism
$$\iota_\ga: \Wfix\lra \fixW.
$$
Let $\cW_{(\Ga)}$ be the image of $\iota_\Ga$;
it is an open and closed substack of $\fixW$. The factored morphism
$\Wfix\to \cW\lrga$ 
is an $\Aut(\Ga)$-torsor.

The cosection localized virtual cycles
$[\cW_{(\ga)}]\virtloc$ are the terms appearing in the localization formula \eqref{van-1}. Because $\cW_\ga\to\cW_{(\ga)}$
is an $\Aut(\ga)$-torsor, the similarly defined virtual cycle $[\cW_{\ga}]\virtloc$ has (\cite[Coro.\,3.8]{CLLL2})
$$[\cW_{\ga}]\virtloc=|\Aut(\ga)|\cdot [\cW_{(\ga)}]\virtloc.
$$

For a vertex $v\in V_\infty(\Ga)$
with $\gamma_v=\{\zeta_5^{a_1},\cdots,\zeta_5^{a_c}\}$, we abbreviate $\gamma_v=(0^{e_0}\cdots 4^{e_4})$,
where $e_i$ is the number of appearances of $i$ in $\{a_1,\cdots,a_c\}$. (We require $a_j\in [0,4]$.)

\begin{defi}
We call a vertex $v\in V_\infty^S(\ga)$ exceptional if $g_v=0$ and $\gamma_v=(1^{2+k}4)$ or $(1^{1+k}23)$,
for some $k\ge 0$.
\end{defi}

\begin{defi}\label{regu}
We call a vertex $v\in V_\infty(\Ga)$ regular if the following hold:
\begin{enumerate}
\item In case $v$ is stable, then either $v$ is exceptional, or for every $a\in S_v$ and $e\in E_v$,
we have $\gamma_a$ and $\gamma_{(e,v)}\in\{\zeta_5,\zeta_5^2\}$.
\item In case $v$ is unstable and $\sC_v$ is a scheme point, then $\sC_v$ is a non-marking
smooth point of $\sC$.
\end{enumerate}
%
%
We call $\Ga$ regular if it is flat, and all its vertices $v\in V_\infty(\ga)$ are regular.
We call $\ga$ irregular if it is not regular.
\end{defi}

Theorem \ref{main} states that
{\sl for a non-pure loop irregular $\Ga$,  $[\cW_{\Ga}]\virtloc\sim 0$.}
We prove an easy and use fact.

\begin{coro}\label{coro}
Let $\Ga\in\Delta^\fl$ be a flat graph that contains an $e\in E_{0\infty}(\ga)$. Suppose $[\cW_{\ga}]\virtloc\ne 0$, then
$d_e=0$ and $\sC_e\cap \sC_\infty$ is a node or a marking of $\sC$.
\end{coro}

\begin{proof}
Since $[\cW_{\ga}]\virtloc\in A\lsta^T (\cW_{\ga}\cap\cW^-)$, $[\cW_{\ga}]\virtloc\ne 0$ implies
that $\cW_{\ga}\cap\cW^-\ne\emptyset$. Let $\xi\in \cW_{\ga}\cap\cW^-$,
then $E_{0\infty}(\Ga_\xi)=\emptyset$. Thus the $e\in E_{0\infty}(\ga)$ must come from flattening a pair of
edges in $E_0(\Ga_\xi)$ and $E_\infty(\Ga_\xi)$. Applying Lemma \ref{unstable-q}, we prove the corollary.
\end{proof}

\section{The Virtual cycle $[\cW_\Ga]\virtloc$}

We begin with recalling the construction of the cosection localized virtual cycle $[\cW_{\ga}]\virtloc$.
Let $\cD$ be the stack of flat families of $(\Sigma^\sC, \sC, \sL,\sN)$, where $\Sigma^\sC\sub \sC$
are pointed twisted curves, $\sL$ and $\sN$ are invertible sheaves on $\sC$. The stack
$\cD$ is a smooth Artin stack, with a forgetful morphism $\cW\to\cD$.
By \cite{CLLL}, we have a perfect relative obstruction theory $\TT_{\cW/\cD}\to \EE_{\cW/\cD}$ and
a cosection $\sigma: \Ob_{\cW/\cD}\to\sO_\cW$.
Letting $\EE_{\cW}=\text{cone}(\TT_{\cD}[-1]\to \EE_{\cW/\cD})$ be the mapping cone, and letting
$\bar\sigma$ be the lift of $\sigma$, which exists. This way, we obtain a perfect obstruction theory
and cosection
$$\phi_{\cW}\dual:\TT_\cW\lra \EE_\cW\and \bar\sigma:\Ob_\cW=H^1(\EE_\cW)\lra\sO_\cW.
\footnote{As argued in \cite[Section 3.1]{CLLL2}, $\phi_{\cW}\dual$ is an arrow in
$D^+_{\mathrm{qcoh}}(\sO_{[\cW/T]})$; and $\bar\sigma$ is $T$-equivariant. (See \cite[Section 3.1]{CLLL2} for notation.)}
$$

Let $\iota_{\ga}: \Wfix\to \cW^T$ be tautological the finite \'etale morphism,
which factor through an $\Aut(\ga)$-torsor $\Wfix\to \cW_{(\ga)}$, with
$\cW_{(\ga)}\sub\cW^T\sub \cW$  open and closed. Taking $\Gm$-fixed part of the obstruction theory of $\cW$,
and using the tautological $\TT_{\cW_{\ga}}\to\TT_\cW$, we obtain
an obstruction theory (c.f. \cite[Prop.\,1]{GP})
\beq\label{ob-ga}\phi_{\cW_{\ga}}\dual : \TT_{\cW_{\ga}}\lra \EE_{\cW_{\ga}}.
\eeq
We then restrict $\iota_{\ga}\sta \bar\sigma$ to the $\Gm$-fixed part of $\iota_{\ga}\sta\Ob_{\cW}$ to obtain
a cosection
$$\iota_{\ga}\sta \bar\sigma^T: \Ob_{\cW_{\ga}}=(\iota_{\ga}\sta\Ob_{\cW})^{\Gm}\lra \sO_{\cW_{\ga}}.
$$
We let
$\cW_{\ga}^-=\cW_{\ga}\cap \cW^-$;
it is the degeneracy locus of $\iota_{\ga}\sta \bar\sigma^T$.

Applying the cosection localized Gysin map in \cite{KL}, we obtain
\beq\label{ab-cW}[\cW_{\ga}]\virtloc=0^!_\loc[\fC_{\cW_{\ga}}]\in A^T\lsta(\cW_{\ga}^-),
\eeq
where $\fC_{\cW_{\ga}}\in h^1/h^0(\EE_{\cW_{\ga}})$ is the intrinsic normal cone.

\medskip
In the remainder of this section, we assume $\Ga$ is an irregular graph with $V_1(\ga)=\emptyset$.
To prove the desired vanishing
$[\cW_\ga]\virtloc\sim 0$, we will work with a construction of $[\cW_{\Ga}]\virtloc$
via obstruction theories of $\cW_{\Ga}$ relative to the auxiliary stack of $\ga$-framed curves
$(\sC,\Si^\sC,\sL,\sN)$ (in $\cD$).

As we will be working with $T$-curves extensively, we make the following convention.
Let $(\Si^\sC,\sC)$ be a pointed $T$-curves, meaning that $T$ acts on the pointed twisted curve
$(\Si^\sC,\sC)$. We denote by $\sC\Tdec$ the curve after decomposing $\sC$ along all its $T$-unbalanced nodes.
Recall that given a flat $\Ga$ and a $(\xi,\eps)\in \cW_\ga$, where $\xi=(\sC,\cdots)$, etc.,
we not only have an identification of the $T$-unbalanced nodes of $\sC$ with $N(\ga)$,
but also an identification of the connected components of $\sC\Tdec$ with
$V^S(\ga)\cup E(\ga)$.
Further the $T$-linearization of $\sC$, and of $(\sL,\sN)$ restricted to each component $\sC_a$ in
$\sC\Tdec$ are specified by the data in $\ga$.

\begin{defi}\label{def-ga}
A $\Ga$-framed (twisted) curve is a $T$-equivariant $(\sC,\Si^\sC,\sL,\sN)$ (in $\cD$),
together with an identification $\eps$ identifying
\begin{enumerate}
\item the marking $\Si^\sC$ with the legs of $\ga$;
\item the $T$-unbalanced nodes of $\sC$ with $N(\ga)$, and
\item the connected components of $\sC\Tdec$ with
$V^S(\ga)\cup E(\ga)$,
\end{enumerate}
so that these identifications are consistent with the geometry of $(\sC,\Si^\sC)$, and
the $T$-linearization of $\sL$ and $\sN$ restricted to each component $\sC_a$ in
$\sC\Tdec$, as specified by the data in $\ga$. 
\begin{enumerate}
\item[(4)]
when $e\in E_{0\infty}(\ga)$, either $\sC_e$ is irreducible and then $\sE_c\cong\Po$, or $\sC_e$ is reducible
and then $\sC_e=\sC_{e-}\cup\sC_{e+}$ is a union
of two $\Po$'s so that, $\sC_{e-}\cap\sC_0\ne\emptyset$, $\sC_{e+}\cap\sC_\infty\ne\emptyset$,
and $\sL\otimes\sK\otimes\bL_1|_{\sC_{e+}}\cong\sO_{\sC_{e+}}$ and $\sN|_{\sC_{e_-}}\cong\sO_{\sC_{e-}}$.
\end{enumerate}
\end{defi}

\begin{defi} \label{def3.1}
A $\ga$-framed gauged twisted curve is a $T$-equivariant data
$\eta=(\sC,\Si^\sC, \sL,\sN, \nu_1,\nu_2)$ with an identification $\eps$
such that
\begin{enumerate}
\item $((\sC,\Si^\sC, \sL,\sN),\eps)\in\cD_\ga$;
\item $(\nu_1,\nu_2)\in H^0(\sL\otimes\sN\otimes\bL_1)^T\oplus H^0(\sN)^T$, such that
$\nu_1|_{\sC_0}=\nu_2|_{\sC_\infty}=0$, and $\nu_1|_{\sC_\infty}$ and $\nu_2|_{\sC_0}$ are nowhere vanishing.;
\item in case of (4) in Definition \ref{def-ga},
$\nu_1|_{\sC_{e+}}$ and $\nu_2|_{\sC_{e_-}}$ are nowhere vanishing.
\end{enumerate}
\end{defi}


Note that the condition (2) are (3) are dictated by (3)-(5) of Definition \ref{def-curve} at the presence of the
fields $(\varphi,\rho)$. Because of (3), The $T$-action on the domain curve of any $\xi\in\cD_\ga$ or $\cDgg$
are completely determined by $\Ga$.

Because the conditions in these definitions are open, we can speak of flat families of $\Ga$-framed curves
and $\ga$-framed gauged curves. We let
$\cD_\ga$ be the stack of flat families of $\Ga$-framed curves, where arrows are $T$-equivariant arrows in
$\cD$ that preserve the data of $\Ga$-framings.
Clearly, $\cD_\ga$ is a smooth Artin stack, with a forgetful morphism
$\cW_\ga\to\cD_\ga$.

We let $\cDgg$ be the stack of flat families of $\ga$-framed gauged twisted curves as in
Definition \ref{def3.1}. It is a smooth Artin stacks.
By forgetting the $\nu$ fields and forgetting the $\varphi$ and $\rho$-fields, respectively,
we obtain the forgetful morphisms $\cDgg\to\cDg$.

Let $\cD_{\ga,[\nu]}\sub\cD_\ga$ be the image stack of the forgetful $\cDgg\to\cD_T$. Let
\beq\label{pp}
\cDgg\mapright{p_1}\cD_{\ga,[\nu]}\mapright{p_2}\cD_\ga
\eeq
be the induced morphisms.

\begin{lemm}
All stacks in \eqref{pp} are smooth.
The morphism $p_1$ is smooth of DM type and the morphism $p_2$ is a closed embedding.
Assuming $V_1(\ga)=\emptyset$, then the fiber dimension of $p_1$ is $|V(\ga)|$,
and the codimension of $\text{image}(p_2)$ is $\sum_{v\in V^S(\Ga)} g_v$.
\end{lemm}

\begin{proof}
The proof that all stacks in \eqref{pp} are smooth, and $p_1$ is smooth and $p_2$ is a closed embedding are
straightforward, and will be omitted. Let $\xi=(\Si^\sC,\sC,\sL,\sN,\nu_1,\nu_2)$ be a close point in
$\cDgg$. In case $V_1(\ga)=\emptyset$, then the fiber dimension of $p_1$ at $\xi$ is the dimension of
choices of locally constant section $\nu_1|_{\sC_\infty}$ and $\nu_2|_{\sC_\infty}$, whose is the number of connected components of
$\sC_0\cup\sC_\infty$, which is $|V_0(\ga)|+|V_\infty(\ga)|=|V(\ga)|$.

The proof of the codimension is similar, and will be omitted.
\end{proof}

We let
$(\cC,\Si^\cC,\cL,\cN,\varphi,\rho,\nu)$ with $\pi: \cC\to\cW_\ga$
be the universal family of $\cW_\ga$. We let
$\cL^{\log}=\cL(-\Si^\cC_{^{(1,\varphi)}})$; let
$\cP^{\log}=\cL^{ -5}\otimes\omega_{\cC/\cW_\ga}^{\log}(-\Si^\cC_{^{(1,\rho)}})$,
and let
\beq\label{cU}
\cU=(\cL^{\log})^{\oplus 5}\oplus \cP^{\log}\and \cV=\cL\otimes\cN\otimes\bL_1\oplus\cN.
\eeq
Using the $\Gm$-invariant version of \cite[Prop.\,2.5]{CL}, the standard relative
obstruction theory of $\cWgg\to\cD_\ga$ is given by
$$\phi_{\cWgg/\cD_\ga}\dual: \TT_{\cWgg/\cD_\ga}\lra \EE_{\cWgg/\cD_\ga}\defeq R\pi_{\ast}^\Gm(\cU\oplus\cV);
$$
the standard relative
obstruction theory of $\cWgg\to\cDgg$ is given by
$$\phi_{\cWgg/\cDgg}\dual: \TT_{\cWgg/\cDgg}\lra \EE_{\cWgg/\cDgg}\defeq R\pi_{\ast}^\Gm\cU.
$$
Like the discussion before \eqref{ab-cW},
paired with their respective standard cosections, we obtain their localized virtual cycles
$[\cWgg]\virt_{\loc,\ga}$ of $\phi_{\cWgg/\cD_\ga}$,
and $[\cWgg]\virt_{\loc,\ga,\nu}$ of $\phi_{\cWgg/\cDgg}$.
Let $\cWgg^-$ be the vanishing locus of the cosection of $\phi_{\cWgg}$, mentioned before \eqref{ab-cW}.
We will show that the vanishing locus of the cosections of $\phi_{\cWgg/\cDg}$ and of $\phi_{\cWgg/\cDgg}$
are identical to $\cWgg^-$.

\begin{prop}\label{idcycle} Let $\Ga$ be irregular. 
Then 
$$[\cW_\Ga]\virt\lloc=[\cW_\Ga]\virt_{\loc,\ga} =[\cW_\Ga]\virt_{\loc,\ga,\nu} \in A\lsta \cW_\Ga^-.
$$
\end{prop}

\begin{proof}
We will choose a relative perfect obstruction theory
$$\phi_{\cWgg/\cD_{\ga,[\nu]}}\dual: \TT_{\cWgg/\cD_{\ga,[\nu]}}\lra \EE_{\cWgg/\cD_{\ga,[\nu]}},
$$
and show that its associated localized virtual cycle $[\cW_\Ga]\virt_{\loc,\ga,[\nu]}$ fits in
the identities
\beq\label{identities}
[\cW_\Ga]\virt_{\loc,\ga,\nu}=[\cW_\Ga]\virt_{\loc,\ga,[\nu]}=[\cW_\Ga]\virt_{\loc,\ga}=
[\cW_\Ga]\virt_{\loc}.
\eeq

We begin with constructing $\phi_{\cWgg/\cD_{\ga,[\nu]}}$.
First, because $p_1:\cDgg\to\cDggg$ is smooth,
$\TT_{\cDgg/\cDggg}$ is a locally free sheaf. Let
\beq\label{T1}
\TT_{\cWgg/\cDgg}\lra \TT_{\cWgg/\cDggg}\lra q\sta\TT_{\cDgg/\cDggg}\mapright{+1}
\eeq
be the d.t. associated with $\cWgg\to\cDgg\to\cDggg$. Here $q$ be the forgetful morphism from
$\cWgg$ to either $\cDgg$ or $\cDggg$, whose meaning will be apparent from the context. We claim that this d.t. splits naturally
via a
\beq\label{tau}
\tau: q\sta\TT_{\cDgg/\cDggg}\to  \TT_{\cWgg/\cDggg}.
\eeq
Indeed, let $\xi\in\cWgg$ be any closed point,
represented by $(\Si^\sC,\sC,\sL,\sN,\varphi,\rho,\nu_1,\nu_2)$. Let $\bar\xi=(\Si^\sC,\sC,\sL,\sN,\nu_1,\nu_2)$
be its image in $\cDgg$. Then any $x\in \TT_{\cDgg/\cDggg}|_{\bar \xi}$ is represented by an extension
$(\ti\nu_1,\ti\nu_2)$ of $(\nu_1,\nu_2)$ as a section of $(\sL,\sN)\times B_2$ over $(\Si^\sC,\sC)\times B_2$, where
$B_2=\spec \CC[\eps]/(\eps^2)$. We define $\tau(\xi)(x)\in \TT_{\cWgg/\cDggg}|_{\xi}$ be the family
$$\bl(\Si^\sC,\sC,\sL,\sN,\varphi,\rho)\times B_2, \ti\nu_1,\ti\nu_2\br.
$$
This definition extends in family version to a homomorphism $\tau$ as in \eqref{tau} that splits \eqref{T1}.
It follows that $q\sta\TT_{\cDgg/\cDggg}\to \TT_{\cWgg/\cDgg}[1]$ is zero, and
\beq\label{sum}
\TT_{\cWgg/\cDggg}=\TT_{\cWgg/\cDgg}\oplus q\sta\TT_{\cDgg/\cDggg}.
\eeq

By the construction of $\cDgg\to\cDggg$, we see that canonically we have
\beq\label{V}q\sta\phi_{\cDgg/\cDggg}\dual: q\sta\TT_{\cDgg/\cDggg}\mapright{\cong} \pi^T\lsta\cV.
\eeq
This together with \eqref{sum} gives us
$$\phi\dual_{\cWgg/\cDggg}=\phi\dual_{\cW_\ga/\cDgg}\oplus q\sta\phi_{\cDgg/\cDggg}:
\TT_{\cWgg/\cDggg}\lra \EE_{\cWgg/\cDggg}\defeq R\pi\lsta^T\cU\oplus \pi\lsta^T\cV
$$
that fits into the following homomorphism of d.t.s:
\beq\label{d.t.1}
\begin{CD}
\EE_{\cW_\ga/\cDgg}@>>>\EE_{\cW_\ga/\cDggg}@>{}>>\pi\lsta^T\cV@>+1>>\\
@AA{\phi\dual_{\cW_\ga/\cDgg}}A@AA{\phi\dual_{\cW_\ga/\cDggg}}A@A{}A{q\sta\phi_{\cDgg/\cDggg}}A\\
\TT_{\cW_\ga/\cDgg}@>>>\TT_{\cW_\ga/\cDggg}@>>>q\sta\TT_{\cDgg/\cDggg}@>+1>>.
\end{CD}
\eeq
Note that by our construction, $\pi\lsta^T\cV$ is a locally free sheaf of rank $|V(\ga)|$.
By inspection, as $q\sta\TT_{\cDgg/\cDggg}$ is a sheaf, we easily see that
$q\sta\phi_{\cDgg/\cDggg}$ is an isomorphism.

We form the following diagram,
\beq\label{d.t.2}
\begin{CD}
\EE_{\cW_\ga/\cDggg}@>{}>>\EE_{\cW_\ga/\cDg}@>>>R^1\pi\lsta^T\cV[-1]@>+1>>\\
@AA{\phi\dual_{\cW_\ga/\cDggg}}A @AA{\phi\dual_{\cW_\ga/\cDg}}A @A{}A{\zeta}A\\
\TT_{\cW_\ga/\cDggg}@>>>\TT_{\cW_\ga/\cDg}@>>>q\sta\TT_{\cDggg/\cDg}@>+1>>,
\end{CD}
\eeq
where the top line is induced by $\pi\lsta^T\cV\to R\pi\lsta^T\cV\to R^1\pi\lsta^T\cV$,
the bottom line is induced via $\cWgg\to \cDggg\to\cDg$. The arrow $\zeta$ is the one making the above
a homomorphism of d.t.s after we show that the left square is commutative.

We now show that the left square in \eqref{d.t.2} is commutative. Using the direct sum \eqref{sum}, and the
definition of $\phi\dual_{\cWgg/\cDggg}$, we see that the desired commutativity follows from the commutativity
of the following two squares:
\beq\label{sq3}
\begin{CD}
\EE_{\cW_\ga/\cDgg}@>>>\EE_{\cW_\ga/\cDg}\\
@AA{\phi\dual_{\cW_\ga/\cDgg}}A@AA{\phi\dual_{\cW_\ga/\cDg}}A\\
\TT_{\cW_\ga/\cDgg}@>>>\TT_{\cW_\ga/\cDg}
\end{CD}\quad\quad
\begin{CD}
\pi\lsta^T\cV@>>>R\pi\lsta^T\cV\\
@AA{q\sta\phi\dual_{\cDgg/\cDggg}}A@AA{\pr_2\circ\phi\dual_{\cW_\ga/\cDg}}A\\
q\sta\TT_{\cDgg/\cDggg}@>{e}>>\TT_{\cW_\ga/\cDg},
\end{CD}
\eeq
where the horizontal arrow $e$ 
is defined via the canonical
$$q\sta\TT_{\cDgg/\cDggg}=
H^0(q\sta\TT_{\cDgg/\cDggg})
\to H^0(q\sta\TT_{\cDgg/\cDg})
\to H^0(\TT_{\cWgg/\cDg})\to \TT_{\cWgg/\cDg}.
$$
We will prove that the left square is commutative in Proposition \ref{functorial}.
For the other, by the construction of the obstruction theory $\phi\dual_{\cW_\ga/\cDg}$, and using that
both $R^i\pi\lsta^T\cV$ are locally free, we conclude that the second square is also commutative.

We also need that $\zeta$ is an isomorphism. We first check that for any closed $\xi\in\cWgg$,
$H^1(\zeta|_\xi)$ is injective. Then using that both $H^1(\TT_{\cDgg/\cDggg})$
and $R^1\pi\lsta^T\cV$ are locally free of identical rank, we conclude that $H^1(\zeta)$ is an isomorphism.
Because $H^{i\ne 1}(\zeta)=0$,
this proves that $\zeta$ is an isomorphism.

Let $\xi$ be any closed point in $\cW_\ga$, represented by $(\Si^\sC,\sC,\sL,\sN,\varphi,\rho,\nu_1,\nu_2)$.
Let $x\ne 0\in H^1(q\sta\TT_{\cDggg/\cDg}|_{\xi})$, which is represented by a first order deformation of
$(\Si^\sC,\sC,\sL,\sN)$ so that $(\Si^\sC,\sC)$ remains constant, $(\sL,\sN)$ is deformed so that
$(\nu_1,\nu_2)$ can not be extended. Then for the same first
order deformation of $(\Si^\sC,\sC,\sL,\sN)$, $(\varphi,\rho,\nu_1,\nu_2)$ does not extend.
Then by \cite[Thm.\,4.5]{BF}, $H^1(\phi_{\cW_\ga/\cD_\ga}\dual|_\xi)(x)$ is the obstruction to the existence of such
extension, thus
$$H^1(\phi_{\cW_\ga/\cD_\ga}\dual|_\xi)(x)\ne 0\in H^1(\EE_{\cW_\ga/\cD_\ga}|_\xi).
$$

On the other hand, because the existence of the extensions of these four fields are independent of each other,
and because extending $(\nu_1,\nu_2)$ is already obstructed, by the construction of the relative
obstruction theory $\phi_{\cW_\ga/\cD_\ga}$,
$$H^1(\zeta|_\xi)(x)=\pr_2\bl H^1(\phi_{\cW_\ga/\cD_\ga}\dual|_\xi)(x)\br\ne 0\in H^1(R^1\pi\lsta^T\cV[-1]|_\xi).
$$
This proves that $H^1(\zeta|_\xi)$ is injective, thus $\zeta$ is an isomorphism.
\smallskip

We now show the first identity in \eqref{identities}; namely $[\cW_\Ga]\virt_{\loc,\ga,\nu}=[\cW_\Ga]\virt_{\loc,\ga,[\nu]}$.
We first apply \cite[Prop.\,2.7]{BF} to \eqref{d.t.1} to obtain a commutative diagram of cone stacks
$$
\begin{CD}
h^1/h^0(\pi\lsta^T\cV[-1])@>>>h^1/h^0(\EE_{\cW_\ga/\cDgg})@>\lam >>h^1/h^0(\EE_{\cW_\ga/\cDggg})\\
@AA{\cong}A @AA{(\phi\dual_{\cW_\ga/\cDgg})\lsta}A@AA{(\phi\dual_{\cW_\ga/\cDggg})\lsta}A\\
h^1/h^0(\pi\lsta^T\cV[-1])@>>>h^1/h^0(\TT_{\cW_\ga/\cDgg})@>>>h^1/h^0(\TT_{\cW_\ga/\cDggg}),\\
\end{CD}
$$
where both rows are exact sequences of abelian cone stacks.
Let
$$\fC_{\cW_\ga/\cDgg}\sub h^1/h^0(\EE_{\cW_\ga/\cDgg})\and
\fC_{\cW_\ga/\cDggg}\sub h^1/h^0(\EE_{\cW_\ga/\cDggg})
$$
be their respective virtual normal cones \cite{BF, LT}.
Applying argument analogous to  \cite[Coro.\,2.9]{CL} (see also \cite[Prop.\,3]{KKP}),
we conclude that $\lam\sta(\fC_{\cW_\ga/\cDggg})=\fC_{\cW_\ga/\cDgg}$.
Because the two cosections of $\Ob_{\cW_\ga/\cDgg}$ and
$\Ob_{\cW_\ga/\cDggg}$ lift to the same cosection of the absolute obstruction sheaf $\Ob_{\cW_\ga}$, we conclude that
the first identity of \eqref{identities} holds.

We prove the second identity of \eqref{identities}. By the same reasoning, from \eqref{d.t.2} we obtain
a commutative diagram of cone stacks
$$
\begin{CD}
h^1/h^0(\EE_{\cW_\ga/\cDggg})@>{\lam'}>>h^1/h^0(\EE_{\cW_\ga/\cDg})@>>>h^1/h^0(R^1\pi\lsta^T\cV[-1])\\
@AA{h^1/h^0(\phi\dual_{\cW_\ga/\cDggg})}A @AA{h^1/h^0(\phi\dual_{\cW_\ga/\cDg})}A @A{}A{\cong}A\\
h^1/h^0(\TT_{\cW_\ga/\cDggg})@>>>h^1/h^0(\TT_{\cW_\ga/\cDg})@>>>h^1/h^0(q\sta\TT_{\cDggg/\cDg}).
\end{CD}
$$
Because $\cDggg\to\cDg$ is a smooth closed embedding of normal bundle $R^1\pi\lsta^T\cV$, $\lam'$ is
a regular embedding of normal bundle isomorphic to $R^1\pi\lsta^T\cV$, where the later is a locally free sheaf.

By the normal cone construction \cite{Ful}, we see that $\lam'\bl h^1/h^0(\EE_{\cW_\ga/\cDggg})\br$
intersects $\fC_{\cWgg/\cDg}$ transversally, and
$\lam^{\prime-1}(\fC_{\cWgg/\cDg})=\fC_{\cWgg/\cDggg}$.
Because the two cosections of $\Ob_{\cW_\ga/\cDggg}$ and
$\Ob_{\cW_\ga/\cDg}$ lift to the same cosection of the absolute obstruction sheaf $\Ob_{\cW_\ga}$, we conclude that
the second identity of \eqref{identities} holds.

\smallskip
Finally, we prove the third identity in \eqref{identities}. From the canonical diagram
\beq\label{candia}
\begin{CD}
\cW_\Ga@>\iota_\ga>>\cW\\
@VV{q}V@VV{\ti q}V\\
\cD_{\ga}@>p>>\cD,
\end{CD}
\eeq
we have the following commutative
\beq\label{WgaDt}
\begin{CD}
\EE_{\cW_\ga/\cD_{\ga}}= ( \iota_\ga\sta\EE_{\cW/\cD})^T@>\sub>> \iota_\ga\sta\EE_{\cW/\cD}\\
@AA{\phi\dual_{\cW_\ga/\cD_{\ga}}}A@AA{\phi\dual_{\cW/\cD}}A\\
\TT_{\cW_\ga/\cD_{\ga}}@>>> \iota_\ga\sta \TT_{\cW/\cD}.
\end{CD}
\eeq
By the construction of the cosection, $\sigma_{\cW_\ga/\cD_{\ga}}=(\iota_\ga\sta\sigma_{\cW/\cD})^T$. Further
\eqref{candia} induces an arrow $q\sta\TT_{\cDg}[-1]\to \TT_{\cWgg/\cDg}$, which composed with
$\phi\dual_{\cW_\ga/\cD_{\ga}}$ in
\eqref{WgaDt} defines the arrow $c$ below:
\beq\label{dia5}
\begin{CD}
@>>>q\sta \TT_{\cD_{\ga}}[-1]@>c>>\EE_{\cWgg/\cD_{\ga}}@>>>\EE_{\cWgg}@>+1>>\\
@.@VV{\eps}V@| @VV{\eps'}V\\
@>>>(\iota_\ga\sta \ti q\sta \TT_{\cD})^T[-1]@>>>( \iota_\ga\sta\EE_{\cW/\cD})^T@>>>
(\iota_\ga\sta\EE_{\cW})^T@>+1>>.
\end{CD}
\eeq
We let $\eps$ be the tautological homomorphism.
By the construction of
$\EE_{\cW_\ga}$ and $\EE_\cW$, both rows are d.t.s. We chose the third vertical arrow
$\eps'$ to be the one making \eqref{dia5} a homomorphism
of d.t.s. It is an isomorphism after $\eps$ is shown to be an isomorphism.

The proof that $\eps$ is an isomorphism can be done with the aid of the stack $\cM$, which is the stack of pointed
twisted nodal curves. Let $\cM_T$ be the stack of pointed twisted nodal curves together with $T$-actions.
Using that the composites
$\cD_\ga\toright{p}\cD\toright{f}\cM$ and $\cD_\ga\toright{h}\cM_T\toright{}\cM$ are identical, we obtain
the following homomorphism of d.t.s:
\beq\label{DMT}
\begin{CD}
(p\sta \TT_{\cD/\cM})^T @>>> (p\sta \TT_{\cD})^T@>>> (p\sta f\sta\TT_{\cM})^T@>+1>>\\
@AA\alpha_1A@AA\alpha_2A@AA\alpha_3A\\
\TT_{\cD_{\ga}/\cM_T} @>>>\TT_{\cD_{\ga}} @>>> h\sta\TT_{\cM_T} @>+1>>.
\end{CD}
\eeq
One then verifies that both $\alpha_1$ and $\alpha_3$ are isomorphisms.
By Five-Lemma, $\alpha_2$ is an isomorphism. This proves that $\eps$ in \eqref{dia5} is an isomorphism.


By \eqref{WgaDt} and \cite[Prop.\,1]{GP}, the composite $\TT_{\cW_\ga}\to\TT_{\cW}\to  \iota_\ga\sta\EE_{\cW}$ lifts to an
obstruction theory $\phi_{\cW_\ga}\dual$, making the following square commutative
\beq
\begin{CD}
\EE_{\cW_\ga/\cD_\ga} @>>> \EE_{\cW_\ga}\\
@AA{\phi_{\cW_\ga/\cD_\ga}\dual}A @AA{\phi_{\cW_\ga}\dual}A\\
\TT_{\cW_\ga/\cD_\ga} @>>> \TT_{\cW_\ga}.
\end{CD}
\eeq
We then take the $H^1$ of the third column in \eqref{dia5} to obtain $\Ob_{\cWgg}{\cong}  (\iota_\ga\sta\Ob_{\cW})^T$.
Further one checks that the two cosections  coincide, which implies that they have identical vanishing
locus $\cW_\Ga^-$.
By the same reasoning as before, we conclude that the localized virtual class $[\cW_\Ga]\virtloc$ defined in \eqref{ab-cW} 
is identical to the class $0^!_{\text{loc}}[\fC_{\cW_\Ga/\cD_\ga}]$ (also see \cite[Prop.\,3]{KKP}).
This proves the lemma.
\end{proof}

\section{The vanishing in no string cases}

We first prove a special case of Theorem \ref{main}.
Let $\Ga\in \Delta^\fl$.
A {\sl string} of $\Ga$ is an $e\in E_{0\infty}(\Ga)$ so that the vertex $v$ of $e$
lying in $V_0(\Ga)$ is unstable and has no other edge attached to it.

\begin{prop}\label{reduction1}
Let $\Ga\in\Delta^\fl$ be irregular and not a pure loop. Suppose it does not contain strings,
then $[\cW_\Ga]\virtloc=0$.
\end{prop}

\begin{rema} \label{flat}
We recall the convention on flat graph. Let $\xi=(\sC,\Si^\sC,\cdots) \in \cWgg$ be any closed point.
Note that $\ga$ might
be different from $\ga_\xi$, which happens when $\ga_\xi$ is not flat, while $\ga$ is the flattening of $\ga_\xi$.
In the case $V_1(\ga)=\emptyset$, then this happens when every $v\in V_1(\ga_\xi)$
has two edges $e_{v-}$ and $e_{v+}$ attached to it, and $\{v,e_{v-},e_{v+}\}$ in $\ga_\xi$ is replaced by a single edge
$e(v)\in E_{0\infty}(\ga)$. Our convention is that $\sC_{e(v)}=\sC_{e_{v-}}\cup\sC_{e_{v+}}$.
\end{rema}

We begin with a special case.

\begin{lemm}\label{red-lem}
Let the situation be as in Proposition \ref{reduction1}. Suppose $V_1(\ga)=\emptyset$ and
$V_0(\Ga)\ne \emptyset$. Then
$[\cW_\Ga]\virtloc=0$.
\end{lemm}

\begin{proof}
Since $\Ga$ will be fixed throughout this proof, for simplicity we will use
$V$, $E$, etc., to denote $V(\Ga)$, $E(\Ga)$, etc.. Recall that for $v\in V^S$($=V^S(\Ga)$),
$E_v$ 
is the set of nodes $\sC_v\cap (\cup_{e\in E} \sC_e)$, and $S_v$ is the set of legs incident to $v$
(cf. \eqref{leg}).

We introduce more notations.
Let $\xi=(\sC,\Si^\sC,\cdots) \in \cWgg$ be a closed point; let $v\in V$.
For $a\in S_v$, in case $\langle\gamma_a\rangle\ne \{1\}$, we let
$m_a\in [1,4]$ be so that $\gamma_a=\zeta_5^{m_a}$.
We let $S_{v}^1\sub S_v$ be the subset of legs decorated by $(1,\varphi)$ or $(1,\rho)$.
(Since $\gamma$ is narrow, no legs are decorated by $1$.)
We denote
$S^1=\cup_{v\in V}S_{v}^1$, and $S^1_\infty=\cup_{v\in V_\infty}S_v^1$, etc.. Similarly, we denote
$S_{v}^{\ne 1}=S_v-S_{v}^{1}$, and
$S^{\ne 1}=\cup_{v\in V} S_{v}^{\ne 1}$. By the definition of MSP fields,
$S_v^{\ne 1}=\emptyset$ when $v\notin V_\infty$, implying $S^{\ne 1}= \cup_{v\in V_\infty}S^{\neq 1}_v$.
\smallskip

We calculate $\vdim\cW_\ga$. Because the perfect obstruction theory of $\cW_\ga$ is that relative
to $\cD_\ga$, we have
\beq\label{dim}
\vdim\cW_\ga=\vdim \cW_\ga/\cD_\ga+\dim\cD_\ga.
\eeq
For the second term, it is
\beq\label{Dgdim}\dim\cD_\ga=\sum_{v\in V^S}(3g_v-3+|E_v|+|S_v|)+\sum_{v\in V^S}2g_v+2h^1(\Ga)-|E|-2.
\eeq
Here $3g_v-3+|E_v|+|S_v|$ represents deformations of $\Si^{\sC_v}\sub \sC_v$, (where $\Si^{\sC_v}$
is defined in \eqref{leg},) $\sum_{v\in V^S}2g_v$ 
represent deformations of $\sL$ and $\sN$ restricting to $\sC_0$ and $\sC_\infty$.
The term $2h^1(\Ga)$ is the deformations of $\sL$ and
$\sN$ contributed by loops in $\Ga$; 
$|E|$ represents automorphisms of $\sC$, and $-2$ is due to the automorphisms of $\sL$ and $\sN$.

Next, using the relative perfect obstruction theory of $\cW_\ga/\cD_\ga$, we know that
$\vdim \cW_\ga/\cD_\ga$ is the sum of \eqref{A} and \eqref{B}:
\beq\label{A}
\chi_T(\sL\otimes\sN\otimes\bL_1)+\chi_T(\sN);
\eeq
\beq\label{B}
\chi_T\bl\sL(-\Sigma^\sC_{(1,\varphi)})^{\oplus 5}\br
+\chi_T\bl\sL^{\vee\otimes 5}\otimes\omega_\sC^{\log}(-\Sigma^\sC_{(1,\rho)})\br.
\eeq

Next, since $\nu_1|_{\sC_\infty}=\nu_2|_{\sC_0}=1$,
as $T$ sheaves $\sN|_{\sC_0}\cong \sO_{\sC_0}$ and $\sL\otimes \sN\otimes\bL_1|_{\sC_\infty}\cong\sO_{\sC_\infty}$.
Let $e\in E_{0\infty}(\Ga)$ be such that the associated curve $\sC_e\cong\Po$, and let $q_0=\sC_e\cap \sC_0$ and
$q_\infty=\sC_e\cap\sC_\infty$. Then using that $\nu_1|_{q_0}=0$, the invariance of $\nu_1$ implies that
$T$ acts non-trivially on $\sC_e$, thus forcing $T$ acts non-trivially on $\sL\otimes \sN\otimes\bL_1|_{q_0}$.
Then as $\nu_2|_{q_\infty}=0$, $T$ acts non-trivially on $\sN|_{q_\infty}$. In case $\sC_e$ is a union of two
$\Po$, a parallel argument shows that the same conclusion holds. Thus as $\Ga$ is connected,
$$\eqref{A}=\chi(\sL\otimes\sN\otimes\bL_1|_{\sC_\infty})+\chi(\sN|_{\sC_0})=
\sum_{v\in V_0}(1-g_v)+\sum_{v\in V_\infty}(1-g_v).
$$
Here when $\sC_v$ is a point, we agree $g_v=0$.

To proceed, we let $\xi=(\sC,\Si^\sC,\cdots) \in \cWgg$ be as before. Let $\chi_T$ of a $T$-sheaf be the $T$-equivariant
$\chi$ of the sheaf. We claim
\begin{align}\label{a}
\chi_T\bl\sL(-\Sigma^\sC_{(1,\varphi)})\br&=\chi\bl\sL(-\Sigma^\sC_{(1,\varphi)})\br.
\end{align}
Indeed, let $v\in V_0(\Ga)$, then because $\varphi|_{\sC_v}\ne 0$ and that $T$-acts trivially on $\sC_v$,
$T$ acts trivially on $\sL|_{\sC_v}$. For the same reason, for $v\in V_\infty(\Ga)$, $T$ acts trivially on both
$\sC_v$ and $\sL|_{\sC_v}$. On the other hand, suppose $E_{0\infty}(\Ga)=\{e\}$ has only one element, with
$\sC_e$ the associated curve. In case $\sC_e\cong\Po$, by Lemma \ref{unstable-q} we have
$\sL|_{\sC_e}\cong\sO_{\sC_e}$. Then \eqref{a} follows. In case $\sC_e$ consists of two $\Po$, say $\sC_e=\sC_{e-}\cup\sC_{e+}$, with $q=\sC_{e-}\cap \sC_{e+}$, $y_{-}=\sC_{e-}\cap\sC_0$ and $y_{+}=\sC_{e+}\cap\sC_\infty$.
By Lemma \ref{unstable-q}, $\deg\sL|_{\sC_{e-}}=-\deg\sL|_{\sC_{e+}}>0$, thus
\begin{align*}
H^i_T\bl\sL(-\Sigma^\sC_{(1,\varphi)})\br=H^i\bl\sL|_{\sC_\infty}(-y_{e+})\br\oplus
H^i\bl\sL(-\Sigma^\sC_{(1,\varphi)})|_{\sC_0}\br,
\end{align*}
and consequently \eqref{a} follows. The case where $E_{0\infty}(\Ga)$ contains many edges is similar. This proves
\eqref{a}.

We next claim that \eqref{a} holds with $\sL(-\Sigma^\sC_{(1,\varphi)})$ replaced by
$\sL^{\vee\otimes 5}\otimes\omega_\sC^{\log}(-\Sigma^\sC_{(1,\rho)})$.
Like before, we first consider the case $E_{0\infty}(\Ga)=\{e\}$.
Because $\ga$ contains no strings, $\sC_e\cap \sC_0$ is a node of $\sC$. By Lemma
\ref{unstable-q}, $\sC_e\cap \sC_\infty$ is also a node of $\sC$.
Thus $\deg \sL^{\vee\otimes 5}\otimes\omega_\sC^{\log}|_{\sC_e}=0$. Then the proof of \eqref{a} shows that
the claim holds in this case. The case $|E_{0\infty}(\Ga)|>1$ can be treated similarly.
This proves the claim.

\medskip

Consequently,
\begin{align*}\eqref{B}&=5\cdot \chi\bl\sL(-\Sigma^\sC_{(1,\varphi)})\br+\chi\bl\sL^{\vee\otimes 5}\otimes\omega_\sC^{\log}(-\Sigma^\sC_{(1,\rho)})\br \\
&=-5|\Sigma^\sC_{(1,\varphi)}|+5\bl\deg \sL+1-g-\sum_{a\in S^{\neq 1}}\frac{m_a}{5}\br+\\
&\qquad+\bl 2g-2+|S|-5\deg \sL-|\Sigma^\sC_{1,\varrho}|+1-g\br\\
&=4(1-g)-4|S_\infty^1| -\sum_{a\in S^{\ne 1}}(m_a-1).
\end{align*}

Because $\Ga$ is bare, $V^S=V^S_0\cup V^S_\infty$; because $\Ga$ has no string, $V^{1,1}_0=\emptyset$.
Therefore $\Sigma_{(1,\rho)}^{\sC}= \cup_{v\in V_0}S_v=\cup_{v\in V^S_0} S_v$.   Similarly, for any
$v\in V^U_\infty$($=V_\infty-V_\infty^S$) that has a leg attached to it, 
$v$ has exactly one edge $e$ attached to it, which must lie in $E_{0\infty}$ as $V_1=\emptyset$.
By Corollary \ref{coro}, the leg of $v$ must be a
scheme marked point (i.e. in $\Sigma_{(1,\varphi)}$).
Thus $S^{\ne 1}= \cup_{v\in V_\infty}S^{\neq 1}_v$ is the same as $\cup_{v\in V_\infty^S}S^{\neq 1}_v$.
Putting together we obtain 
\beq\label{aux}\sum_{v\in V^S}|S_v|=|\Sigma_{(1,\rho)}^\sC|+|S^{\ne 1}|  +\sum_{v\in V^S_\infty} |S^1_v|.
\eeq

\noindent
{\bf Assumption I}.
{\sl No leg of $\Ga$ is decorated by $(1,\rho)$,
	and $m_a\ne 1$ for every $a\in S^{\ne 1}$. }
\smallskip

Under this assumption, we have the Euler equation $|E|-|V|=h^1(\Ga)-1$,
$g=\sum_{v\in V^S}g_v+h^1(\Ga)$, and $\Si^\sC_{(1,\rho)}
=\emptyset$. Using \eqref{dim}, and adding \eqref{Dgdim}, \eqref{A} and \eqref{B}, we get
\beq\label{dimens}
\vdim \cWgg= \bl \sum_{v\in V^S}|E_v| -4|S_\infty^1| +\sum_{v\in V^S_\infty} |S^1_v|\br
-3(|E|-|V^U|)  -\sum_{a\in S^{\ne 1}}(m_a-2) .
\eeq
Note that when $\Ga$ is a pure loop, it is zero.
We now prove that under the assumption of the proposition,
\eqref{dimens} is negative when $[\cW_{(\ga)}]\virtloc\ne 0$, impossible. 

We first consider the case where $V^S=\emptyset$. Since $\Ga$ is not a pure loop,
$\sC$ is a chain of $\Po$ connecting two vertices $v$ and $v'$.
Since $V_1=\emptyset$, $E=E_{0\infty}$. Since $\Ga$ has no strings, both $v$ and $v'\in V_\infty$.
Then by Corollary \ref{coro}, each $v$ and $v'$ has one leg in $\Sigma_{(1,\varphi)}$ attached to it.
Thus $|S^1_\infty|=2$ and $|E|-V^U|=-1$, implying that \eqref{dimens} is $-4\cdot 2+3<0$.


We now assume $V^S\ne \emptyset$. Out strategy is to divide the contribution in \eqref{dimens} by
looking at the maximal simple chains in
$\Ga$. Here a simple chain in $\Ga$ consists of distinct edges $E_1,\cdots, E_k$ and vertices $v_0,\cdots, v_k$
so that $E_i$ has vertices $v_{i-1}$ and $v_i$, and $v_{0<i<k}$ are unstable. Since   $\Ga$ is not
a pure loop and $|V^S|>0$, if $\{E_1,\cdots,E_k\}$ is a maximal simple chain in $\Ga$, then one of $\{v_0,v_k\}$ must
be stable.

Clearly maximal simple chains give  partitions of $E$ and $V^U$. Now let
$\{E_1,\cdots,E_k\}$
be a maximal simple chain in $\Ga$. Suppose $v_0$ is stable but $v_k$ is not,
then $v_k\in V_\infty^U$  because $\Ga$ contains no strings. Thus $|S_{v_k}^1|=1$ by Corollary \ref{coro}.
Therefore the contribution to \eqref{dimens} from $\{E_1,\cdots E_k,v_1,\cdots,v_k\}$ is
\beq\label{chain1}1-4|S_{v_k}^1|=-3.
\eeq
The other case is when the maximal chain has both $v_0$ and $v_k$
stable, then  the contribution to \eqref{dimens} from $\{E_1,\cdots E_k,v_1,\cdots,v_{k-1}\}$ is
\beq\label{chain2}2-3=-1.\eeq

We now show that \eqref{dimens} is non-positive. Let $\ga'$ be the graph resulting from removing
all edges, all unstable vertices, and all legs attached to unstable vertices.
Because every $e\in E$ or $v\in V^U$ is contained in exactly one maximal simple chain,
the previous argument shows that
$$\vir.\dim \cW_\ga\le \vir.\dim\cW_{\ga'}.
$$
Applying the formula \eqref{dimens} to $\vir.\dim\cW_{\ga'}$, we see that it contains terms of the following kind:
(i) terms associated to elements in $\cup_{v\in V^S}S^1_v$; each contributes $-4+1=-3$; (ii)
terms associated to elements in $ \cup_{v\in V^U} S^1_v$; each contributes $-4$;  (iii)
terms associated to elements $a\in S^{\neq 1}$; since $m_a\geq 2$ by our {\sl simplifying assumption},
it contributes $m_a-2\leq 0$. This shows that \eqref{dimens} is $\leq 0$.

When \eqref{dimens} is zero, we must have $E=V^U=S^1=\emptyset$, and $m_a=2$ for every
$a\in S^{\neq 1}$. But this is impossible because $\Ga$ is irregular. This proves that under the {\sl simplified assumption} and
when $\Ga$ is not a pure loop, $\text{vir}.\dim \cW_\Ga<0$, implying $[\cW_\Ga]\virtloc=0$.
\medskip

We now prove the proposition without assuming {\sl Assumption I}.
We first suppose $\Ga$ has a leg $i_0$ ($i_0$-th leg) decorated by $\gamma_{i_0}=\zeta_5$ attached to
$v\in V_\infty$. We claim that $v$ is stable. Indeed, otherwise $v$ has an edge $e$ attached to $v$; since
$\ga$ is bare $e\in E_{0\infty}(\ga)$; but by Lemma \ref{unstable-q} $d_e\in \ZZ$, which contradicts to
that $i_0$ is decorated by $\zeta_5$. Thus $v$ is stable.
We let $\Ga'$ be the graph obtained by removing the leg $i_0$ from $\Ga$, except the following cases,
when $\ga$ is a one vertex graph, $g_v=0$ and $|S_v|=3$. (Note that since $\Ga$ is irregular, $g_v=1$ and $|S_v|=1$
is impossible.)
Note that if $g_v=0$, $|S_v|+|E_v|=3$, and $v$ has at least one edge, $v$ in $\Ga'$ becomes unstable.

Following \cite[Thm.\, 4.5]{CLL}, we have a forgetful morphism
$$\cF: \cW_\Ga\lra\cW_{\Ga'}
$$
that send
$\xi=(\Si^\sC,\sC,\cdots)\in\cW_\Ga$ to $\xi'=(\Si^{\sC'},\sC', \cdots)\in\cW_{\Ga'}$ by
{\sl marking forgetting and stabilizing}.

\noindent
{\bf Marking forgetting and stabilizing}. {\sl The curve
$\sC'$ is from $\sC$ by forgetting the marking $\Si^\sC_{i_0}$, making $\sC$ scheme along
$\Si^\sC_{i_0}$, and stabilize if necessary, with $\sC'$ the resulting curve; $\Si^{\sC'}$ is $\Si^\sC$ with
$\Si^\sC_{i_0}$ deleted; letting $\eps: \sC\to\sC'$ the resulting morphism, letting
$\sL'=\eps\lsta\sL$ and $\sL'\otimes\sN'=\eps\lsta(\sL\otimes \sN)$, while $\varphi'$, etc., is the
pushforward of $\varphi$, etc., respectively.}

We next compare the virtual cycles $[\cW_\ga]\virtloc$ and $[\cW_{\Ga'}]\virtloc$.
First, applying {\sl Marking forgetting and stabilizing},
we obtain a morphism $f: \cD_{\ga}\to\cD_{\ga'}$, which
fits into the commutative square shown:
\beq\label{fg-square}
\begin{CD}
\cW_\Ga @>{\cF}>> \cW_{\Ga'}\\
@VVV @VVV\\
\cD_{\ga} @>{f}>>\cD_{\ga'}.
\end{CD}
\eeq

Let $(\cC,\Si^\cC,\cL,\cdots)$ and $(\cC',\Si^{\cC'},\cL',\cdots)$ be the universal families of $\cW_\ga$ and
$\cW_{\ga'}$, respectively, with $\pi:\cC\to\cW_\ga$ and $\pi':\cC'\to\cW_{\ga'}$ their projections.
The stabilization defines the $\Psi$ below
$$\begin{CD}
\cC @>\Psi >> \cC'\times_{\cW_{\ga'}}\cW_\ga @>\pr >> \cC'\\
@VV\pi V @VV\ti\pi V @VV\pi' V\\
\cW_\ga @= \cW_\ga @>\cF >> \cW_{\ga'}.
\end{CD}
$$
It is direct to check that we have canonical isomorphisms $\pr\sta\cL'\cong \Psi\lsta \cL$, and
$R^1\Psi\lsta \cL=0$. This implies
$$R\pi\lsta^T\cL(-\Si^\cC\lophi)\cong R\ti\pi\lsta^T \Psi\lsta \cL(-\Si^\cC\lophi)
\cong \cF\sta R\pi^{\prime T}\lsta \cL'(-\Si^{\cC'}\lophi),
$$
and similar isomorphisms with $\cL(-\Si^\cC\lophi)$ replaced by $\cL^{-5}\otimes \omega_{\cC/\cW_\ga}(-\Si^\cC\lorho)$,
etc..

Like in \cite{CLL}, this shows that the relative obstruction theory of $\cW_{\Ga'}\to\cD_{\ga'}$
pullbacks to that of $\cW_\Ga\to\cD_{\ga}$,
and the cosection of $\Ob_{\cW_{\Ga'}/\cD_{\ga'}}$
pullbacks to that of $\Ob_{\cW_{\Ga}/\cD_\ga}$. Thus,
letting $\theta=\cF|_{\cW_\ga^-}: \cW_\Ga^-\to \cW_{\Ga'}^-$, we have
\beq\label{iso00}
\theta\sta[\cW_{\Ga'}]\virtloc=[\cW_{\Ga}]\virtloc.
\eeq

In case $\Ga$ has a leg decorated by $(1,\rho)$, we remove this leg from $\Ga$, resulting a new
graph $\Ga'$. (In this case since $\Ga$ is irregular, $\Ga$ can not be a single vertex graph.)
Then we have a similarly defined forgetful morphism $\cF: \cW_\Ga\to\cW_{\Ga'}$ (with stabilization
if necessary) and $\theta$ as before so that \eqref{iso00} holds.

By repeating this procedure (of removing legs labeled by $\zeta_5$ or $(1,\rho)$) we obtain a
graph $\Ga'$ and morphisms $\cF$ and $\theta$ as before so that \eqref{iso00} holds.
As $\Ga'$ is bare, not a pure-loop and satisfies the {\sl Assumption I}, we have $[\cW_{\Ga'}]\virtloc=0$.
By \eqref{iso00}, $[\cW_{\Ga}]\virtloc=0$. This proves the lemma.
\end{proof}

\begin{proof}[Proof of Proposition \ref{reduction1}]
By a result proved at the end of \cite[Section 3]{CLLL2}, we know $[\cW_\ga]\virtloc=0$ if
there is a $v\in V^{0,2}_\infty(\ga)$ so that the two edges $e$ in $\ga$ incident to $v$ both lie
in $E_\infty(\ga)$ and have $d_e=\deg\sL|_{\sC_e}\in\ZZ$. We now suppose $\ga$ has no such
vertices.

We next trimming all edges of $\ga$ in $E_0(\ga)\cup E_\infty(\ga)$. For $e\in E_0(\ga)$, in case $e$ is incident to a
stable $v\in V_0^S(\ga)$, or in case $e$ is incident to an unstable $v\in V_0^U(\ga)$ so that another edge in
$E(\ga)$ is also incident to
$v$, we then remove $e$ and add a new leg decorated by $(1,\rho)$ and attached it to $v$; otherwise we remove $e$, $v$, and
any other legs incident to $v$.

For $e\in E_\infty(\ga)$, in case $e$ is incident to a
stable $v\in V_\infty^S(\ga)$, or in case $e$ is incident to an unstable $v\in V_0^U(\ga)$ so that another edge in
$E(\ga)$ is also incident to
$v$, we then remove $e$ and add a new leg decorated by $\gamma_{(e,v)}$\footnote{We assign $\gamma_{(e,v)}=(1,\varphi)$
in case $d_e\in\ZZ$, otherwise $\gamma_{(e,v)}=e^{-2\pi\sqrt{-1} d_e}$ (cf. before Defi.\,\ref{graph2}).}
and attached it to $v$; otherwise we remove $e$, $v$, and
any other legs incident to $v$. After performing these operations to all $e$ in $E_0(\ga)$ and $E_\infty(\ga)$, and after
discarding all vertices in $V_1(\ga)$, we obtain a new graph $\ga'$. Let $\{\ga_i\}$ be the connected components of
$\ga'$.

Applying the discussion \cite[Section 3]{CLLL2} to this situation, we conclude that if $[\cW_{\ga'}]\virtloc=0$,
then $[\cW_\ga]\virtloc=0$. By our assumption on $\ga$, we know that all $\ga_i$ in $\{\ga_i\}$
are non loop and bare; at least one such $\ga_i$ is irregular. Because
$$[\cW_{\ga'}]\virtloc=\prod[\cW_{\ga_i}]\virtloc,
$$
applying Lemma \ref{red-lem}, we have that $[\cW_{\ga'}]\virtloc=0$. This proves the proposition.
\end{proof}

\begin{coro}\label{van34}
In case $\Ga$ consists of a single stable vertex $v\in V_\infty(\Ga)$ 
such that its legs are decorated by
$\gamma_1,\cdots,\gamma_\ell\in\bmu_5-\{1\}$ and that at least one $\gamma_i\in
\{\zeta_5^3,\zeta_5^4\}$, then $[\cW_\Ga]\virtloc=0$, except when $g_v=0$ and
$\gamma=(1^{1+k}23)$ or $=(1^{2+k}4)$, for a $k\ge 0$.
\end{coro}

\section{Reduction to no-string cases}
\def\ba{{\mathbf a}}
\def\cDggpri{\cD_{\ga',\nu}}

The proof of the general case is by reduction to no-string cases.
To this end, we introduce the operation {\sl trimming a leaf edge} from a graph.\footnote{A leaf edge
is an edge so that one of its vertex is unstable and has only one edge attached to it.}

\begin{defi}\label{-string}
Let $\Ga\in\Delta^\fl$ and let $e\in E_{0\infty}(\Ga)$ be a string (thus a leaf edge).
Let $v_-\in V_0(\Ga)$ and $v_+\in V_\infty(\Ga)$ be its vertices.
The trimming of $e$ from $\Ga$ is by first removing $e$,  $v_-$ and all legs attached to $v_-$, and
then attaching a leg, called the distinguished leg, decorated by $(1,\varphi)$ to $v_+$.
\end{defi}

In the following, we will use the induction on the number of strings to prove Theorem \ref{main}.
We fix a $\Ga$ with a string $e$, and its two associated vertices $v_\pm$, as in Definition \ref{-string}.
We assume $[\cW_{\ga}]\virtloc\ne 0$, and shall derive a contradiction in the end.
We denote by $\ga'$ the graph after trimming $e$ from $\Ga$.

Like before, let $\cDgg$ be the stack of $\ga$-framed gauged curves $((\sC,\Si^\sC,\sL,\sN,\nu), \eps)$.
For any family $(\cC, \Si^{\cC}, \cL,\cN,\nu)$ (with $\eps$ implicitly understood) in $\cDgg$,
because $e$ is a string of $\Ga$, the correspondence $a=(e,v_+)\in F(\Ga)$
associates to a section of nodes $\cR_{a}\sub\cC$ that splits off a family of rational curves $\cC^e\sub\cC$
(associated with $e$),
called the $e$-tail of $\cC$.\footnote{ Fibers of $\cC^e$ can be one $\Po$, or a union of two $\Po$'s.
See Remark \ref{flat}.}
We let
$$\cC^\diamond=\overline{\cC-\cC^e}\sub \cC
$$
be the complement of $\cC^e$ in $\cC$.

We consider the family
\beq\label{ecirc}
(\cC^\diamond, \Si^\cC\cap \cC^\diamond+ \cR_a,\cL|_{\cC^\diamond},\cN|_{\cC^\diamond},\nu|_{\cC^\diamond}).
\eeq
Together with the induced framing, it is a family in $\cDggpri$. As this construction is canonical, we obtain a
forgetful morphism
$$\cDgg\lra \cDggpri.
$$

We need another stack, of elements in $\cDgg$ paired with fields on its $e$-tail.
Given $(\sC,\Si^\sC,\sL,\sN,\nu)\in\cDgg$, we abbreviate
$$\sL^{\log}=\sL(-\Si^\sC\lophi),\and \sP^{\log}=\sL^{\vee\otimes 5}\otimes\omega_\sC^{\log}(-\Si^\sC\lorho).
$$

\begin{defi}\label{partial-e} Let $(\sC,\Si^\sC,\sL,\sN,\nu)\in\cDgg$. 
A $(\varphi,\rho)$-field on its $e$-tail is 
$$(\varphi^e,\rho^e)=(\varphi^e_1,\cdots,\varphi^e_5,\rho^e)\in H^0(\sL^{\log}|_{\sC^e})^{\oplus 5}\oplus
H^0(\sP^{\log}_{\sC^e}).
$$
A partial $e$-field on a $\ga$-framed gauged curve consists of a
$\zeta\in \cDgg$ and a $(\varphi,\rho)$-field on its $e$-tail.
\end{defi}

We let $\cY_{\ga,\nu,e}$ be the groupoid of families of partial $e$-fields on $\ga$-framed gauged curves.
That is, elements in $\cY_{\ga,\nu,e}$ are $(\sC,\Si^\sC,\sL,\sN,\nu, \varphi^e,\rho^e)$ (with the $\ga$-framing
implicitly understood) as in Definition
\ref{partial-e}.

Let $(\cC, \Si^{\cC}, \cL,\cN,\varphi,\rho,\nu)$ to be the
universal family on $\cW_\ga$. Like before, the flag $a=(e,v_+)\in F(\Ga)$
associates to a section of nodes $\cR_{a}\sub\cC$, which splits $\cC$ into two
subfamilies $\cC^e$ and $\cC^{e\circ}$.
The family
$$(\cC, \Si^{\cC}, \cL,\cN,\nu,\varphi|_{\cC^e},\rho|_{\cC^e})
$$
then is a family in $\cYgg$, which induces a forgetful morphism
$\delta: \cWgg\to \cYgg$. Of course, by forgetting the fields on the $e$-tail, we obtain a
forgetful morphism $\zeta: \cYgg\to \cDgg$.

To proceed, we let $\bar e$ be the graph which is $e$ with two vertices $v_-$ and $v_+$,
together with the decorations on $e$ and the legs on $v_-$ (if any), plus a new leg
decorated by $1=\zeta_5^0$ attached to $v_+$. Note that because of the decoration $1$, $\bar e$
is of broad type
(cf. the first paragraph in section 2.1).

We let $\cWge$ and  $\cWee$ be the moduli stack of stable
$\ga'$ and $\bar e$-framed MSP fields, respectively. By restricting the universal family of $\cYgg$ to its
$e$-tails, we obtain a family on $\cC^e$, which induces a morphism $\cYgg\to \cWee$.
We list these morphisms together:
\beq\label{00} \cW_\ga\mapright{\delta} \cYgg\lra \cWee
\and \cYgg\mapright{\zeta} \cDgg.
\eeq

What we would like to have is that via restricting the universal family of $\cW_\ga$ to $\cC^\diamond$ we
obtain a family in $\cW_{\Ga'}$, thus getting a morphism from $\cW_\ga$ to $\cW_{\ga'}$.
Unfortunately, this in general is not true because $\varphi|_{\cR_{a}}$
might not vanish identically, thus does not necessarily induces a morphism
$\cW_\ga\to\cW_{\ga'}$. (Recall $\cR_a$ associates to a marking of $\ga'$ labeled by $(1,\varphi)$.)

To remedy this, we let $\cWee^\mu=(\cWee)\lred$ be $\cWee$ with the reduced stack structure; let
\beq\label{WY}
\cYgg^\mu=\cYgg\times_{\cWee}\cWee^\mu, \and \cW_\ga^\mu=\cW_\ga\times_{\cWee}\cWee^\mu.
\eeq

\begin{lemm}\label{factII}
The stack $\cWee$ has pure dimension four; it has hypersurface singularities, and is acted on by
the group $GL(5,\CC)$. The coarse moduli of $\cWee^\mu=(\cWee)_\redd$ is isomorphic to $\Pf$, and the
induced $GL(5,\CC)$ action on this coarse moduli is the standard $GL(5,\CC)$ action on $\Pf$.
\end{lemm}

\begin{proof}
We begin with classifying the closed points
of $\cWee$. Let $\xi=(\sC,\Si^\sC,\cdots)\in \cWee$ be a closed point, let $\Ga_\xi$ be its associated
graph. We claim that $\Ga_\xi\ne \Ga_\xi^\fl$. Indeed, in case $\Ga_\xi= \Ga_\xi^\fl$, then
$\sC\cong \Po$ and $T$ acts on $\sC$ with two fixed points, $p_-$ and $p_+$, associated with
the vertices $v_-\in V_0(\Ga_\xi)$ and $v_+\in V_\infty(\Ga_\xi)$, respectively.
Because we have assumed that $[\cW_\ga]\virtloc\ne 0$, by Corollary \ref{coro}, we have
$\deg\sL=0$. Since $p_+$ is a marking decorated by $1$, and $p_-$ is either a non-marking
or a marking decorated by $(1,\rho)$, we have $\omega_\sC^{\log}(-\Si^{\sC}_{(1,\rho)})\cong
\sO_{\Po}(-1)$, forcing $\rho=0$, contradicting to $\rho|_{p_+}\ne 0$. This proves $\Ga_\xi\ne \Ga_\xi^\fl$.

In case $\Ga_\xi\ne \Ga_\xi^\fl$, it contains two edges: $e_+\in E_{\infty}(\ga_\xi)$ and
$e_-\in E_{0}(\Ga_\xi)$. Let $\sC_{\pm}\sub\sC$ be the irreducible component associated with $e_\pm$.
Then $\sC=\sC_-\cup \sC_+$ with one node $q$, associated with the vertex in $V_1(\Ga_\xi)$. Let $p_\pm
\in \sC_\pm\sub\sC$ be the two $T$ fixed points (other than $q$), as before. Then by the definition of MSP fields, we have
$\sN|_{\sC_-}$ and $\sL\otimes\sN|_{\sC_+}$ are trivial. Adding $\deg\sL=0$ and $\deg\sN=c$, where $c=
d_{\infty e}$, we get $\sL|_{\sC_-}\cong \sO_{\sC_-}(c)$, $\sL|_{\sC_+}\cong \sO_{\sC_+}(-c)$
and $\sN|_{\sC_+}\cong \sO_{\sC_+}(c)$. Consequently,
\beq\label{van10}
\varphi|_{\sC_+}=\rho|_{\sC_-}=0,
\eeq
and because $\varphi|_{p_-}$ and $\rho|_{p_+}$ are non-trivial,
$$ \varphi|_{\sC_-}\in H^0(\sO_{\sC_-}(c)^{\oplus 5})^T-0\cong \CC^5-0,\quad
\rho|_{\sC_+}\in H^0(\sO_{\sC_-})^T-0\cong \CC-0.
$$
Adding that $\nu_1$ and $\nu_2$ are
non-trivial and unique up to scaling ($T$-equivariant), we see that $\xi$ is uniquely parameterized by
$$[\varphi_1(p_-),\cdots,\varphi_5(p_-)]\in\Pf.
$$
Repeating a family version of this argument, we prove that
the coarse moduli of $\cWee^\mu$ is isomorphic to $\Pf$.

The mentioned $GL(5,\CC)$ action on $\cWee$ is the obvious one. Given any family
in $\cWee$, which is given by $(\cC,\Si^\cC,\cL,\cN, \varphi, \rho,\nu)$, we define
$\sigma\cdot (\cC,\Si^\cC,\cL,\cN, \varphi, \rho,\nu)$ to be
$(\cC,\Si^\cC,\cL,\cN, \sigma\cdot\varphi, \rho,\nu)$, where
$\sigma\cdot\varphi$ is the standard matrix multiplication after viewing $\varphi$ as a column vector
with exponent $\varphi_i$, and viewing $\sigma$ as a $5\times 5$ invertible matrix. This defines a
$GL(5,\CC)$ action on $\cWee$, and its action on the coarse moduli of $(\cWee)\lred\cong\Pf$
is the standard action of $GL(5,\CC)$ on $\Pf$.

Finally, we prove that $\cWee$ has hypersurface singularity. For this, we first calculate
the tangent space and the obstruction space of $\cWee$ at its closed points.
Let $\xi=(\sC,\Si^\sC,\sL,\cdots)$ be a closed point of $\cWee$. As argued before, $\sC=\sC_-\cup\sC_+$,
with $\deg\sL|_{\sC_\pm}=\mp c$ for a $c\in\ZZ_+$, $\deg\sN|_{\sC_-}=0$ and $\deg\sN|_{\sC_+}=c$.
A direct calculation shows that
$$H^1(\sL\otimes\sN\otimes\bL_1)^T=H^1(\sN)^T=H^1(\sL^{\log})^T=0,\and H^1(\sP^{\log})=\CC.
$$
This shows that the obstruction space to deformations of $\xi\in \cWee$ is always one dimensional.
Because $\cWee$ has pure dimension $4$, we conclude that
$\dim T_\xi\cWee=5$, that $\cWee$ is locally defined by
one equation in a smooth $5$-fold, and thus $\cWee$ has hypersurface singularities.
\end{proof}

We now compare the stacks $\cW_\ga^\mu$, $\cYgg^\mu$, etc.. We first show that the family \eqref{ecirc} together with
$(\varphi,\rho)|_{\cC^\diamond}$ 
defines a morphism
\beq\label{pi1}\cW_\ga^\mu\defeq \cW_\ga\times_{\cYgg}\cYgg^\mu\lra \cW_{\ga'}.
\eeq
Indeed, by the prior discussion, it suffices to show that
\beq\label{11}\varphi|_{\cR_{a}\times_{\cYgg}\cYgg^\mu}=0.
\eeq
By the vanishing $\varphi|_{\sC_+}=0$ in \eqref{van10}, the
$\varphi$-field of any closed $\xi\in \cW_{\bar e}$ restricted to $v_+\in\sC$ vanishes.
This shows that \eqref{11} holds, and the morphism \eqref{pi1} exists.

Next, by definition the composite morphism $\cWgg^\mu\to \cWgg\to 
\cWee$ (cf. \eqref{00})
factors through $\cWgg^\mu\to \cWee^\mu$.
Paired with \eqref{pi1}, we obtain a morphisn $\beta$ as shown: 
\beq\label{Diag}
\begin{CD}
\cW_\ga^\mu @>\beta>> \cWge\times \cWee^\mu \\ 
@VVV@VVV\\ 
\cYgg^\mu @>\beta' >> \cD_{\ga',\nu}\times\cWee^\mu \\
\end{CD}
\and
\begin{CD}
\cYgg @>{R_e}>> \cW_{\bar e}\\
@VVV @VVV\\
\cDgg @>{r_{e}}>> \cD_{\bar e,\nu}.
\end{CD}
\eeq
The other arrows in \eqref{Diag} are as follows.
Let $\cC_{\cYgg}$ be the domain curve of the universal family of $\cYgg$. Because curves in the family
$\cC_{\cYgg}$ are $\ga$-framed, it contains a distinguished section of nodes
$\cR_a\sub\cC_{\cYgg}$, where $a=(e,v_+)$, which splits out the $e$-tails $\cC_{\cYgg}^e$ of
$\cC_{\cYgg}$.
The universal family of $\cYgg$ restricted to $\cC_{\cYgg}^e$ induces the morphism
$R_e: \cYgg\to\cWee$. The similar construction gives $r_e$, as shown. 
Next, by removing the $\varphi^e$ and $\rho^e$ from the universal family of $\cYgg$ and then restricting the
remainder part to $\cC_{\cYgg}^\diamond$, we obtain a family in $\cD_{\ga',\nu}$, which defines a morphism
$\cYgg^\mu\to \cD_{\ga',\nu}$. Paired with the tautological $\cYgg^\mu\to\cWee^\mu$, we obtain the $\beta'$
in \eqref{Diag}. By constructions, these two squares are commutative.

\begin{lemm}\label{lem5.4}
The horizontal arrows in \eqref{Diag} are smooth; The morphisms $\beta$ is a $\bmu_5$-torsors,
and the square involving $R_e$ and $r_e$ is Cartesian.
\end{lemm}

\begin{proof}
We prove that $\beta$ is a $\bmu_5$-torsor. Following the construction, we see that $\beta$ surjective.
We now show that it is a $\bmu_5$-torsor. Indeed, given any closed point
$$z=((\sC',\Si^{\sC'},\sL,\cdots),(\sC^e,\Si^{\sC^e},\sL^e,\cdots))\in \cW_{\ga'}\times\cWee^\mu,
$$
any point in $\beta\upmo(z)$ is by gluing $\sC'$ and $\sC^e$ along the markings in $\sC'$ and $\sC^e$
associated to $(e_+,v_+)$, and gluing the $\sL$'s and $\sN$'s on $\sC'$ and $\sC^e$.
As the marking is a scheme point, the
gluing of markings is unique. Because the section $\nu_1$ is non-vanishing at the markings, the gluing of
$\sL\otimes\sN$ is also unique. On the other hand, the gluing of $\sL$ is constrained by the non-vanishing of
$\rho$'s. As $\rho$ restricted to the marking to be glued, it is a section of $\sL^{\vee\otimes 5}$. Thus
the gluing of $\sL$ are unique up to $\bmu_5$. As this argument works for family,
this shows that $\beta$ is a $\bmu_5$-torsor.

The other conclusions can be proved similarly, and will be omitted.
\end{proof}

Following \cite[Prop.\,2.5]{CL} as before,
we endow $\cWge$ and $\cWee$ their tautological perfect relative obstruction theories, relative to
$\cDgg$ and $\cD_{\bar e}$, respectively. For $\cWee$, as it is proper by Lemma \ref{factII}, we let $[\cWee]\virt\in A\lsta \cWee$
be its virtual class. For $\cWge$, like $\cW_\ga$, we form its standard cosection
$\sigma_{\ga',\nu}: \Ob_{\cWge/\cD_{\ga',\nu}}\to \sO_{\cWge}$,
which is liftable to a cosection of $\Ob_{\cWge}$. Let
$\cW^-_{\ga'}\sub \cWge$ be its degeneracy locus (with reduced structure),
and let $[\cWge]\virtloc\in A\lsta \cW^-_{\ga'}$ be its
associated cosection localized virtual class.

We let
$$\cW_\ga^{\sim}=\cW_\ga^\mu\times_{\kappa, \cW_{\ga'}}\cW_{\ga'}^-\sub \cW_\ga^\mu,
$$
where $\kappa: \cW_\ga^\mu\mapright{\beta} \cW_{\ga'}\times\cWee^\mu\mapright{\pr}\cW_{\ga'}$
is the composite. We let
\beq\label{kappa}
\ti\kappa: \cW_\ga^{\sim}\lra \cW^-_{\ga'}
\eeq
be induced by $\kappa$. Because $\beta$ is a $\bmu_5$-torsor, $\ti\kappa$ is flat. Because $\cWee^\nu$ is
proper, $\kappa$ is a proper morphism.

\begin{prop}\label{reduction} The stack $\cW_\ga^\sim$ is proper, and contains $\cW_\ga^-$ as
its closed substack. Let $\jmath: \cW_\ga^{-}\to \cW_\ga^\sim$ be the inclusion.
Then there is a rational $c\in\QQ$ such that
$$\jmath\lsta[\cW_\ga]\virtloc=c\cdot \ti\kappa\sta [\cW_{\ga'}]\virtloc\in A\lsta(\cW^\sim_{\ga}).
$$
\end{prop}


We prove Theorem \ref{main} assuming Proposition \ref{reduction}.

\begin{proof}[Proof of Theorem \ref{main}]
Let $\Ga\in\Delta^\fl$ be irregular and not a pure loop. In case $\ga$ has no strings, then the vanishing
follows from Proposition \ref{reduction1}.

Now assume $\Ga$ has strings. 
Let $e$ be a string of $\ga$, and let $\ga'$ be the result after trimming $e$ from $\ga$.
In case $\Ga'=\emptyset$,
by Lemma \ref{unstable-q}, the marking $\sC_{v_+}$ is a scheme marking of type $(1,\phi)$. Thus
$\vir\dim \cW_\ga =\vir\dim\cW_{\bar e} -5 =3-5< 0$, implying $[\cW_\ga]\virtloc=0$.

Otherwise $\ga'\in\Delta^\fl$ is irregular, not a pure loop, and has one less string than that of $\ga$.
Thus by induction, we have $[\cWge]\virtloc\sim 0$.
By Proposition \ref{reduction}, we get $\jmath\lsta[\cWgg]\virtloc\sim 0$.
Namely, there is a proper substack $\cZ'$, $\cW_{\ga'}^-\sub \cZ'\sub\cW_{\ga'}$, so that the cycle
$[\cWge]\virtloc$ pushed to $A\lsta \cZ'$ is zero. Let $\cZ=\kappa\upmo(\cZ')$.
Since $\kappa$ is proper, $\cZ$ is proper. Also, $\cW_\ga^-\sub\cZ$. Then Theorem \ref{reduction}
implies that the pushforward of $[\cW_\ga]\virtloc$ to $A\lsta\cZ$ is zero. This proves
$[\cWgg]\virtloc\sim 0$.
\end{proof}

\section{Proof of Proposition \ref{reduction}}

We continue to denote by $\delta:\cW_\ga\to\cYgg$ the (representable) morphism induced by restriction.
The relative obstruction theory of $\cYgg\to\cDgg$ pullback to $\cWgg$ takes the form
$$\delta\sta\phi\dual_{\cYgg/\cDgg}: \delta\sta\TT_{\cYgg/\cDgg}\lra \delta\sta \EE_{\cYgg/\cDgg}=R\pi_{\ast}^{\Gm}
(\cU|_{\cC^e}).
$$
Here $\cC^e$ and $\cC^\diamond\sub\cC$ are the two families of subcurves (of the universal curve
$\cC$ of $\cW_\ga$) after decomposing along $\cR_{a}$, where $a=(e,v_+)$; $\cU$ is defined in
\eqref{cU}. We let
$$\EE_{\cWgg/\cYgg}=R\pi_{\ast}^\Gm \bl \cU|_{\cC^\diamond}(-\cR_a)\br.
$$

%
%
%
%

Recall $\EE_{\cW_\ga/\cDgg}= R\pi\lsta^T\cU$.
Using the exact sequence $\cU|_{\cC^\diamond}(-\cR_a)\to \cU\to \cU|_{\cC^e}$
and the pair $\delta: \cW_\ga\to\cYgg$, we form the top and the bottom d.t.s
\beq\label{dia1}
\begin{CD}
\EE_{\cWgg/\cYgg}@>\alpha>> \EE_{\cWgg/\cDgg} @>\beta>>\delta\sta \EE_{\cYgg/\cDgg}@>+1>>\\
@AA{\ti\phi\dual_{\cWgg/\cYgg}}A@AA{\phi\dual_{\cWgg/\cDgg}}A@AA{\delta\sta\phi\dual_{\cYgg/\cDgg}}A\\
\TT_{\cWgg/\cYgg}@>\ti\alpha>>  \TT_{\cWgg/\cDgg} @>\ti\beta>>\delta\sta \TT_{\cYgg/\cDgg} @>+1>>
\end{CD}
\eeq
where the second and the third vertical arrows are the perfect obstruction theories constructed by direct image cones, and the square is commutative because of Proposition \ref{functorial0}. We let
$\ti\phi\dual_{\cWgg/\cYgg}$ be the one making \eqref{dia1} a morphism of d.t.s. Applying Five-Lemma,
it is also a perfect obstruction theories.


%
%
%
\smallskip

Let $\sigma_{\ga,\nu}$ be the cosection of $\Ob_{\cWgg/\cDgg}$ mentioned after Definition \ref{def3.1}; let
\beq\label{tilsi}
\ti\sigma_{\ga,\nu}\defeq \sigma_{\ga,\nu}\circ H^1(\alpha):\Ob_{\cWgg/\cYgg}\lra \sO_{\cWgg}.
\eeq

\begin{lemm}\label{deg-comp}
The degeneracy locus 
$D(\ti\sigma_{\ga,\nu})=\{\xi\in\cW_\ga\mid \ti\sigma_{\ga,\nu}|_\xi=0\}$ 
is proper. 
\end{lemm}

\begin{proof}
The construction of $\sigma_{\ga,\nu}$ is as in \cite{CLL}, where it is proved that it lifts to $\bar\sigma_{\ga,\nu}
: \Ob_{\cW_\ga}\to\sO_{\cW_\ga}$ (cf.
\cite[Prop.\,3.4]{CLL}).

We now show that $D(\ti\sigma_{\ga,\nu})$ is proper.
Let $\xi\in \cWgg$ be a closed point, represented by $\xi=(\sC,\Si^\sC,\cdots,\nu)$. Let
$\sR_a\sub\sC$ be the node associated with $a=(e,v^+)\in F(\ga)$, which decomposes $\sC$ into subcurves
$\sC^\diamond$ and $\sC^e$.
By the description of the obstruction theory of $\cWgg\to\cYgg$, 
$$\Ob_{\cWgg/\cYgg}|_\xi=H^1\bl \sL^{\log}|_{\sC^\diamond}(-\sR_a)^{\oplus 5}\oplus
\sP^{\log}|_{\sC^\diamond}(-\sR_a)\br^T,
$$
where $\sL^{\log}$ and $\sP^{\log}$ are as defined before \eqref{cU}.

Let
$$\xi^\diamond\defeq (\sC^\diamond, \Si^{\sC^\diamond}=\Si^{\sC}\cap\sC^\diamond+\sR_a,
\sL|_{\sC^\diamond},\cdots, \nu_2|_{\sC^\diamond}),
$$
where the marking $\sR_a$ is decorated by $(1,\varphi)$. Then $\xi^\diamond$ is a point in $\cWge$.
Following the construction of the obstruction theory of $\cW_\ga/\cD_{\ga,\nu}$,
we see that
$$\Ob_{\cWge/\cD_{\ga',\nu}}|_{\xi^\diamond}=
H^1\bl \sL|_{\sC^\diamond}(-\Si_{(1,\varphi)}^{\sC^\diamond})^{\oplus 5}\oplus
\sL^{\vee\otimes 5}|_{\sC^\diamond}\otimes\omega_{\sC^\diamond}^{\log}(-\Si_{(1,\rho)}^{\sC^\diamond})\br^T.
$$
Because of the identities
$$\sP|_{\sC^\diamond}=
\sL|_{\sC^\diamond}^{\vee\otimes5}\otimes\omega_{\sC^\diamond}^{\log},\quad
\Si^\sC_{(1,\rho)}|_{\sC^\diamond}=\Si^{\sC^\diamond}_{(1,\rho)},\and
\Si^\sC_{(1,\varphi)}|_{\sC^\diamond}+\sR_a=\Si^{\sC^\diamond}_{(1,\varphi)},
$$
we have
$\sL(-\Si_{(1,\varphi)}^\sC)|_{\sC^\diamond}(-\sR_a)=\sL|_{\sC^\diamond}(-\Si_{(1,\varphi)}^{\sC^\diamond})$, and
the exact sequence
\beq\label{inj3}
0\lra
\sP^{\log}
|_{\sC^\diamond}(-\sR_a)\mapright{\text{}}
\sL|_{\sC^\diamond}^{\vee \otimes 5}\otimes\omega_{\sC^\diamond}^{\log}(-\Si_{(1,\rho)}^{\sC^\diamond})
\lra \sP^{\log}|_{\sR_a}\lra 0.
\eeq
Therefore we get the induced surjective
\beq\label{cinc}
r: \Ob_{\cWgg/\cYgg}|_\xi \mapright{} \Ob_{\cWge/\cD_{\ga',\nu}}|_{\xi^\diamond}.
\eeq
By the definition of the cosections $\sigma_{\ga,\nu}|_\xi$ and $\sigma_{\ga',\nu}|_{\xi^\diamond}$,
we see that (cf. \eqref{tilsi})
$$\begin{CD}
\Ob_{\cWgg/\cYgg}|_\xi @>{r}>> \Ob_{\cWge/\cD_{\ga',\nu}}|_{\xi^\diamond}\\
@VV{\ti\sigma_{\ga,\nu}|_\xi}V @VV{\sigma_{\ga',\nu}|_{\xi^\diamond}}V \\
\CC @= \CC
\end{CD}
$$
is commutative.  
Therefore, $\ti\sigma_{\ga,\nu}|_\xi=0$ implies that
$\kappa(\xi)\in D(\sigma_{\ga',\nu})$. (cf. $\kappa:\cW_\ga^\mu\to \cW_{\ga'}$ is defined before \eqref{kappa}).
This proves that
$$D(\ti\sigma_{\ga,\nu})\sub\kappa\upmo(D(\sigma_{\ga',\nu})).
$$
As $D(\sigma_{\ga',\nu})$ is proper (\cite{CLLL}) and $\kappa$ is proper,
$D(\ti\sigma_{\ga',\nu})$ is proper.
\end{proof}

We let
$$\fC_{\cWgg/\cYgg}\sub h^1/h^0\bl \TT_{\cWgg/\cYgg}\br\sub \fA_{\ga,e}\defeq h^1/h^0(\EE_{\cWgg/\cYgg})
$$
be the virtual normal cone (cf. \cite{BF}).
Following \cite{KL}, the cosection $\ti\sigma_{\ga,\nu}$ defines a bundle stack homomorphism
$\ti\sigma_{\ga,\nu}: \fA_{\ga,e}\to \sO_{\cWgg}$.
We let $\fA_{\ga,e}(\ti\sigma_{\ga,\nu})\sub\fA_{\ga,e}$ be the kernel stack of $\ti\sigma_{\ga,\nu}$,
which is a closed substack of $\fA_{\ga,e}$ defined via
\beq\label{tisi2}
\fA_{\ga,e}(\ti\sigma_{\ga,\nu})
\defeq \coprod_{\xi\in\cW_\ga}\ker\{\ti\sigma_{\ga,\nu}|_\xi: \fA_{\ga,e}|_\xi\lra\CC\},
\eeq
endowed with the reduced stack structure.

\begin{lemm}\label{deg-van}
We have
$( \fC_{\cWgg/\cYgg})_\redd\sub \fA_{\ga,e}(\ti\sigma_{\ga,\nu})$.
\end{lemm}

\begin{proof}
Let
$$\fC_{\cWgg/\cDgg}\sub h^1/h^0\bl \TT_{\cWgg/\cDgg}\br\sub \fA_\ga \defeq h^1/h^0(\EE_{\cWgg/\cDgg})
$$
be the similarly defined virtual normal cone. By \cite{KL},
$$(\fC_{\cWgg/\cDgg})_\redd \sub \fA_\ga(\sigma_{\ga,\nu}),
$$
where $\fA_\ga(\sigma_{\ga,\nu})\sub \fA_\ga$ is the kernel stack of $\sigma_{\ga,\nu}$.
Applying the functoriality of the $h^1/h^0$ construction to \eqref{dia1},
we obtain the commutative diagram
\beq\label{h1h0}
\begin{CD}
\fC_{\cWgg/\cYgg}@>\subset>>h^1/h^0(\TT_{\cWgg/\cYgg})@>\subset>>\fA_{\ga,e}=h^1/h^0(\EE_{\cWgg/\cYgg})\\
@VVV @VVV@VV{h^1/h^0(\alpha)}V \\
\fC_{\cWgg/\cDgg}@>\subset>>h^1/h^0(\TT_{\cWgg/\cDgg}) @>\subset>> \fA_\ga=h^1/h^0(\EE_{\cWgg/\cDgg}).\\
\end{CD}
\eeq
Because $(\fC_{\cWgg/\cDgg})_\redd\sub \fA_\ga(\sigma_{\ga,\nu})$, via the definition of $\ti\sigma_{\ga,\nu}$
(cf. \eqref{tilsi}) we conclude
$(\fC_{\cWgg/\cYgg})_\redd\sub \fA_{\ga,e}(\ti\sigma_{\ga,\nu})$.
\end{proof}


Our next step is to use the virtual pullback of  \cite[Def.\,2.8]{CKL} (also \cite[Constr.\,3.6]{Man}) to
re-express the cycle $[\cWgg]\virtloc$.
For this, we need a description of the virtual normal cone of $\cYgg\to \cDgg$:
\beq\label{Be}
\fC_{\cYgg/\cDgg}\subseteq h^1/h^0\bl \TT_{\cYgg/\cDgg}\br\subseteq {\fB}\defeq h^1/h^0(\EE_{\cYgg/\cDgg}).
\eeq
We show that the identities in \eqref{Be} hold.

Indeed, by Lemma \ref{factII}, $\cWee$ has pure dimension $4$, equaling the expected dimension of
$\cWee$, and has local complete intersection singularities, the intrinsic normal cone $\fC_{\cWee/\cD_{\bar e,\nu}}$
equals the bundle stack $\fA_{\bar e}$, shown below.
\beq\label{Fe}
\fC_{\cWee/\cD_{\bar e,\nu}}=\fA_{\bar e}\defeq h^1/h^0(\EE_{\cWee/\cD_{\bar e,\nu}}).
\eeq
Because the second square in \eqref{Diag} is a Cartesian
square, we have
\beq\label{BBB}
\fC_{\cYgg/\cDgg}=\fC_{\cWee/\cD_{\bar e,\nu}}\times_{\cWee}\cYgg=\fA_{\bar e}\times_{\cWee}\cYgg
=\fB.
\eeq

\smallskip

We form Cartesian products and projections as shown
\beq\label{pipi}
\begin{CD}
\bAB\defeq \fA_{\ga,e}\times_{\cYgg}\fB @>{\pi_2}>> \cW_{\ga|\fB}\defeq \cW_\ga\times_{\cYgg}\fB @>>> \fB\\
@VVV@VV{\pi_1}V @VVV\\
\fA_{\ga,e} @>{\beta}>> \cW_\ga @>>> \cYgg.
\end{CD}
\eeq
Note that $\pi_2$ is the pullback of $\beta$ via $\pi_1$.
Viewing $\ti\sigma_{\ga,\nu}: \fA_{\ga,e}\to \sO_{\cWgg}$ as a bundle-stack homomorphism, its pullback
$$\pi_1\sta(\ti\sigma_{\ga,\nu}): \bAB\lra \sO_{\cW_{\ga|\fB}}
$$
is a bundle-stack homomorphism too. 
Its degeneracy locus then is 
\beq\label{WWW}
D(\pi_1\sta(\ti\sigma_{\ga,\nu}))
=D(\ti\sigma_{\ga,\nu})\times_{\cYgg}{\fB}\sub \cW_{\ga|\fB},
\eeq
and its associated kernel stack $\bAB(\pi_1\sta\ti\sigma_{\ga,\nu})$ (cf. \eqref{tisi2}) is
$$\bAB(\pi_1\sta\ti\sigma_{\ga,\nu})=
\fA(\ti\sigma_{\ga,\nu})\times_{\cWgg} \cW_{\ga|\fB}\sub \bAB.
$$
We denote the inclusion by $\iota$:
\beq\label{iota2}
\iota: (\fC_{\cW_{\ga|\fB}/{\fB}})\lred \sub (\fC_{\cWgg/\cYgg}\times_{\cYgg} \fB)\lred \sub
\bAB(\ti\sigma_{\ga,\nu}),
\eeq
where the first inclusion follows from the definition of $\cW_{\ga|\fB}$, and the second follows from Lemma
\ref{deg-van}.

To proceed, let us recall the virtual pullbacks introduced in \cite{Man}. 
Following \cite{Man}, we form the composite:
\beq\label{fsha}
f^!: A\lsta \fB\mapright{\eps}A\lsta \bAB\mapright{0^!_{\pi_2}} A\lsta \cW_{\ga|\fB}\mapright{0^!_{\pi_1}}
A\lsta \cW_\ga.
\eeq
Here the arrow $\eps$ is defined as follows. Let
$\bar\eps': Z\lsta\fB\to Z\lsta(\fC_{\cW_{\ga|\fB}/{\fB}})$ be the linear map defined
via $\bar\ep'([V]) =[\fC_{V\times_{\fB}\cW_{\ga|\fB}/V}]$.
Since $\cWgg$ is a DM stack,
both $\cWgg\to \cYgg$ and $\cW_{\ga|\fB}\to {\fB}$ are of DM type. Applying the proof of
\cite[Thm.\,2.31]{Man} to
\cite[Constr.\,3.6]{Man}, we conclude that $\bar \eps'$ descends to the $\bar\eps$ in \eqref{sha}.
Let $\bar\iota\lsta:A\lsta \fC_{\cW_{\ga|\fB}/{\fB}} \to
A\lsta\bAB$ be induced by the inclusion \eqref{iota2}. We define $\eps$ be the composite
\beq\label{sha}
\eps: A\lsta\fB\mapright{\bar\eps} A\lsta(\fC_{\cW_{\ga|\fB}/{\fB}})\mapright{\bar\iota\lsta} A\lsta\bAB.
\eeq
The arrows $0^!_{\pi_1}$ and $0^!_{\pi_2}$ in \eqref{fsha} are
Gysin maps after intersecting with the zero sections of the bundle stacks $\pi_1$ and $\pi_2$, respectively.
\smallskip

The version we will use is the localized analogue of \eqref{fsha}:
\beq\label{shaloc}
f^!\lloc: A\lsta\fB\mapright{\ti\ep}
A\lsta(\bAB(\pi\sta_1\ti\sigma_{\ga,\nu}))\mapright{0^!_{\pi_2,\loc}} A_{\ast}(D(\pi_1\sta\ti\sigma_{\ga,\nu}))
\mapright{0^!_{\pi_1}} A_{\ast}(D(\ti\sigma_{\ga,\nu})).
\eeq
By \eqref{iota2}, the $\ep$ in \eqref{fsha} (cf. \eqref{sha}) factors through
$A\lsta(\bAB(\ti\sigma_{\ga,\nu}))$, giving the $\ti\eps$ in \eqref{shaloc}.
Since $D(\pi\sta_1 \ti\si_{\ga,\nu})$ is proper, the last arrow $0^!_{\pi_1}$ is the ordinary Gysin map of the
bundle-stack $\pi_1$.

\begin{prop}\label{vpb} 
Let $\jmath: D(\sigma_{\ga,\nu})\to D(\ti\sigma_{\ga,\nu})$ be the inclusion, then
$$f\lloc^![\fC_{\cYgg/\cDgg}]=\jmath\lsta [\cWgg]\virtloc\in A\lsta (D(\ti\sigma_{\ga,\nu})).
$$
\end{prop}

\begin{proof}
We  quote the relative version of
cosection localized pullback in \cite[Prop.\,2.11]{CKL}, stated in \cite[Remark 2.12]{CKL}. The proof of
\cite[Prop.\,2.11]{CKL} carries word by word to our case, such as $\cWgg/\cYgg$ satisfies the ``virtually smooth" condition in
\cite[(2.1)]{CKL} because of \eqref{dia1}. The cosections setup are also consistent. Proposition \ref{vpb} follows.
\end{proof}

We are ready to prove Proposition \ref{reduction}. We let
$$\fA_{\bar e}^\mu=\fA_{\bar e}\times_{\cWee}\cWee^\mu\and
\fB^\mu=\fB\times_{\cYgg}\cYgg^\mu.
$$
By Lemma \ref{factII}, $\fA_{\bar e}^\mu$ is a bundle stack over $\cW_{\bar e}$, where the later is irreducible.
Thus for a rational number $c$, $[\fA_{\bar e}]=c\cdot [\fA_{\bar e}^\mu]$.
Because the second square in \eqref{Diag} is Cartesian, using \eqref{BBB}, we conclude that
$$[\fB]=[\fC_{\cYgg/\cDgg}]=
[\fA_{\bar e}\times_{\cWee}\cYgg]=c\cdot [\fA_{\bar e}^\mu\times_{\cWee}\cYgg]=
c\cdot[\fB^\mu].
$$
Therefore by \eqref{BBB},
\beq\label{bbb}
f^!\lloc([\fC_{\cYgg/\cDgg}])=f^!\lloc([\fB])=c\cdot f^!\lloc([\fB^\mu]).
\eeq

Let $\kappa: \cWgg^\mu\to \cW_{\ga'}$ be induced by the $\beta$ (in \eqref{Diag}); let
$\ti\kappa: \cW_\ga^\sim\to \cW_{\ga'}^-$ be that induced by $\kappa$, as defined in \eqref{kappa}.
Let
$$\theta: \fA_{\ga,e}|_{\cWgg^\mu}=h^1/h^0(\EE_{\cWgg/\cYgg})|_{\cWgg^\mu}\lra \kappa\sta h^1/h^0(\EE_{\cWge/\cD_{\ga',\nu}}),
$$
be induced by \eqref{inj3} and the identity before \eqref{inj3}; it is a smooth morphism. We claim that (as cycle)
\beq\label{cycle}
[\fC_{\cW^\mu_\ga/\cYgg^\mu}]= \theta\sta [\fC_{\cWge/\cD_{\ga',\nu}}]\in
Z\lsta \bl h^1/h^0(\EE_{\cW^\mu_\ga/\cYgg^\mu})\br.
\eeq

To this end, we introduce a new stack $\cD_{\ga',\nu,\diamond}$, consisting of objects
$(\xi,\rho_\diamond)$, where $\xi=(\cC,\Sigma^\cC,\cL,\cN,\cdots)\in \cD_{\ga',\nu}(S)$ and
a nowhere vanishing $\rho_\diamond\in H^0(\omega_{\cC}^{\log}\otimes \cL^{\vee\otimes 5})|_\sR)$,
where $\cR\sub\cC$ is the section of the marking associated with the distinguished $1_\varphi$-leg of $\Ga'$.
(The distinguished leg the added one after trimming the edge $e$; see definition \ref{-string}.)

For any family $(\cC,\Sigma^\cC,\cL,\cN,\nu, \phi^e,\rho^e)$ in $\cYgg^\mu(S)$,
we let $\cR\sub\cC$ be section of nodes that separate $\cC$ into $\cC^\diamond$ and $\cC^e$
(cf. \eqref{ecirc}). Then by adding $\rho|_\cR$ to the family \eqref{ecirc} we obtain a family in
$\cD_{\ga',\nu,\diamond}$. This defines the morphism $\zeta_1$ below.
Let $\alpha$ shown below be the morphism defined similarly.
They form the (left) commutative diagram
\beq\label{Xdia}
\begin{CD}
\cW^\mu_\ga @>\kappa>>\cW_{\ga'}@>=>>\cW_{\ga'}\\
@VVV@VV{\alpha}V@VVV\\
\cYgg^\mu@>\zeta_1>>\cD_{\ga',\nu,\diamond}@>\zeta_2>>\cD_{\ga',\nu}
\end{CD}
\eeq
Let $\zeta_2$ be the forgetful morphism. It fits into the right commutative diagram above.
Because for family $(\cC,\cdots,\phi^e,\rho^e)$ in $\cYgg^\mu(S)$,
$\phi^e|_\cR=0$, one hecks directly that the left square above is a fiber product.

By its construction, $\zeta_2$ is smooth. Thus
 \beq\label{Xdia0}
 \begin{CD}
 \fC_{\cWge/\cD_{\ga',\nu,\diamond}}@>>> \fC_{\cWge/\cD_{\ga',\nu}}\\
@VVV@VVV\\
h^1/h^0(\TT_{\cWge/\cD_{\ga',\nu,\diamond}})@>>>h^1/h^0(\TT_{\cWge/\cD_{\ga',\nu}})
 \end{CD}
 \eeq
is a fiber product.
This implies
\beq\label{aux}  \TT_{\cW^\mu_\ga /\cYgg^\mu} \cong \kappa\sta  \TT_{\cW_{\ga'}/\cD_{\ga',\nu,\diamond}}
\and \fC_{\cW^\mu_\ga /\cYgg^\mu} \cong \kappa\sta  \fC_{\cW_{\ga'}/\cD_{\ga',\nu,\diamond}}.
\eeq
By \eqref{Xdia0} and \eqref{aux}, the following square is a fiber product:
\beq\label{diaX3}\begin{CD}
  \fC_{\cW^\mu_\ga/\cYgg^\mu} @>>>    \kappa\sta \fC_{\cWge/\cD_{\ga',\nu}}  \\
 @VVV@VVV\\
  h^1/h^0(\TT_{\cW^\mu_\ga/\cYgg^\mu}) @>>> \kappa\sta h^1/h^0(\TT_{\cWge/\cD_{\ga',\nu}}).
\end{CD}
\eeq

We next look at their deformation complexes.
To begin with, the family version of \eqref{cinc}
gives an exact sequence
\beq
\kappa\sta\alpha\sta \TT_{\cD_{\ga',\nu,\diamond}/\cD_{\ga',\nu}}
\lra \Ob_{\cW^\mu_\ga/\cYgg^\mu}\lra \kappa\sta\Ob_{\cWge/\cD_{\ga',\nu}}\lra 0.
\eeq
Note that $\TT_{\cD_{\ga',\nu,\diamond}/\cD_{\ga',\nu}}$ is an invertible sheaf whose fibers
are $(\omega_\sC^{\log}\otimes \sL^{\vee\otimes5})|_\sR$.
This sequence is the cohomology of the top row in
 \beq\label{diaX}
 \begin{CD}
 \TT_{\cD_{\ga',\nu,\diamond}/\cD_{\ga',\nu}}[-1]
@>>> \EE_{\cW^\mu_\ga/\cYgg^\mu}@>>> \kappa\sta\EE_{\cWge/\cD_{\ga',\nu}}@>+1>>\\
@| @AA{\phi\dual_{\cW^\mu_\ga/\cYgg^\mu}}A@AA{\phi\dual_{\cWge/\cD_{\ga',\nu}}}A\\
 \TT_{\cD_{\ga',\nu,\diamond}/\cD_{\ga',\nu}}[-1]
@>>> \TT_{\cW^\mu_\ga/\cYgg^\mu}@>>> \kappa\sta\TT
_{\cWge/\cD_{\ga',\nu}}@>+1>>.
\end{CD}
\eeq
Here the upper row is induced by derived push-forward of the family version of \eqref{inj3};
the lower row is  induced by \eqref{Xdia} and \eqref{aux}. Hence both rows are distinguished triangles. The   arrow  $\phi\dual_{\cWge/\cD_{\ga',\nu}}$ is induced by the ordinary   construction
and $\phi\dual_{\cW^\mu_\ga/\cYgg^\mu}$ is induced  by
the same process deriving  the first  vertical arrow in \eqref{sq-A} using   \eqref{blcone}'s blow up construction. Both vertical arrow uses direct image cone constructions. The commutativity of the second square in \eqref{diaX} follows from the natural arrow between  two universal families and two evaluations maps directly.


Taking $h^1/h^0$ of the diagram we obtain
\beq\label{diaX2}
 \begin{CD}
 T_{\cD_{\ga',\nu,\diamond}/\cD_{\ga',\nu}}
@>>> h^1/h^0(\EE_{\cW^\mu_\ga/\cYgg^\mu})@>\theta>> \kappa\sta h^1/h^0(\EE_{\cWge/\cD_{\ga',\nu}})\\
@|@AA{h^1/h^0(\phi\dual_{\cW^\mu_\ga/\cYgg^\mu})}A@AA{ h^1/h^0(\phi\dual_{\cWge/\cD_{\ga',\nu}})}A\\
 T_{\cD_{\ga',\nu,\diamond}/\cD_{\ga',\nu}}
@>>>  h^1/h^0(\TT_{\cW^\mu_\ga/\cYgg^\mu})@>>> \kappa\sta h^1/h^0(\TT
_{\cWge/\cD_{\ga',\nu}}).
\end{CD}
\eeq
 By \cite[Prop.\,2.7]{BF}, both rows are exact sequences of cone stacks. Therefore the second square of \eqref{diaX2} is a fiber product. By Proposition \ref{functorial0},
we know that $\ti\phi\dual_{\cW^\mu_\ga/\cYgg^\mu}$(in \eqref{dia1}) is $\nu$-equivalent to
$\phi\dual_{\cW^\mu_\ga/\cYgg^\mu}$ (cf. \cite[Def.\,2.9]{semi}), thus the cycle
$[\fC_{\cW^\mu_\ga/\cYgg^\mu}]$ 
induced by $\ti\phi\dual_{\cW^\mu_\ga/\cYgg^\mu}$ is identical to that induced by
$\phi\dual_{\cW^\mu_\ga/\cYgg^\mu}$ (cf. \cite[Prop.\,2.10]{semi} and \cite[Lemm.\,2.3]{semi}). Combined with \eqref{diaX3}, this proves the clam \eqref{cycle}.

We consider $\pi_1^\mu$ (, compare with $\pi_1$ in \eqref{pipi},)
$$\pi_1^\mu\defeq \pi|_{\cW_{\ga|\fB}^\mu}: \cW_{\ga|\fB}^\mu\defeq \cW_\ga^\mu\times_{\cYgg}\fB=\cWgg^\mu\times_{\cWee}\fA_{\bar e}^\mu
\lra \cWgg^\mu,
$$
where $\fA_{\bar e}^\mu=\fA_{\bar e}\times_{\cWee}\cWee^\mu$.
We let
$$\psi: \cW_{\ga|\fB}^\mu= 
\cWgg^\mu\times_{\cWee}\fA_{\bar e}^\mu\lra \cWgg^\mu
$$
be the first projection. Then

Then by the definition of $\ti\eps$ (cf. \eqref{shaloc} and \eqref{sha}), 
$$\ti\eps[\fB^\mu]
=[ \fC_{\cWgg^\mu/{\cYgg^\mu}} \times_{\cWgg^\mu}\fA_{\bar e}^\mu]
=\psi\sta \theta\sta[\fC_{\cWge/\cD_{\ga',\nu}} ].
$$
Applying $0^!_{\pi_1\sta\ti\sigma_{\ga,\nu},\text{loc}}$, we obtain 
$$ 0^!_{\pi_1\sta\ti\sigma_{\ga,\nu},\text{loc}} \bl\psi\sta \theta\sta[\fC_{\cWge/\cD_{\ga',\nu}} ]\br
=\psi\sta \ti\kappa\sta\bl 0^!_{\ti\sigma_{\ga',\nu},\text{loc}}[\fC_{\cWge/\cD_{\ga',\nu}} ]\br
=\psi\sta \ti\kappa\sta [\cW_{\ga'}]\virtloc.
$$
(Recall $\ti\kappa: \cW_\ga^{\sim}\to \cW^-_{\ga'}$ is defined in \eqref{kappa}.)
Adding
$$0^!_{\pi_1} (\psi\sta \ti\kappa\sta[\cW_{\ga'}]\virtloc)
=\ti\kappa\sta[\cW_{\ga'}]\virtloc\in A\lsta \cW_\ga^\sim=A\lsta D(\ti\sigma_{\ga,\nu}),
$$
we prove that
$$f^!\lloc[\fB^\mu]=\ti\kappa\sta[\cW_{\ga'}]\virtloc\in A\lsta D(\ti\sigma_{\ga,\nu}).
$$
This proves Proposition \ref{reduction}.

\section{Appendix}
\def\ccX{\mathcal X}

Let $\ccX$ be an Artin stack; let $\pi:\cC\to \ccX$ be a flat family twisted nodal curves, and
let $\cV\to \cC$ be a smooth morphism of quasi-projective type. We denote by $C(\pi\lsta \cV)$ the groupoid
defined as follows: for any
scheme $S$, $C(\pi\lsta \cV)(S)$ consists of all
$(\sigma, s)$, where $\sigma: S\to\ccX$ is a morphism, $\cC_\sigma=\cC\times_\ccX S$ and
$\cV_\sigma=\cV\times_\cC\cC_\sigma$, 
and $s: S\to \cV_\sigma$ is an $S$-morphism (a section of $\cV_\sigma\to S$).
Arrows between two objects $(\sigma, s)$ and $(\sigma', s')$
consists of an arrow between $\sigma$ and $\sigma'$ so that $s=s'$ under the induced isomorphism
$\cV_\sigma\cong\cV_{\sigma'}$.

We abbreviate $\cW=C(\pi\lsta\cV)$.
Let $\pi_\cW:\cC_\cW \to \cW$
be the pullback of $\cC\to\ccX$, and let $\ev:\cC_\cW\to \cV$ be the tautological
evaluation map (induced by the section $s$), which fits into the commutative diagrams
\beq\label{first-Q}
\begin{CD}
\cW@<\pi_\cW<<\cC_{\cW}@>{\ev}>>\cV\\
@VVV@VVV@VVV\\
\ccX@<<<\cC@=\cC.
\end{CD}
\eeq
Applying the projection formula to
$\pi_{\cW}^\ast \TT_{\cW/\ccX}\cong \TT_{\cC_\cW/\cC}\to\ev^\ast \TT_{\cV/\cC}$,
and using
$\TT_{\cW/\ccX}\lra R\pi_{\cW\ast}\pi_\cW\sta\TT_{\cW/\ccX}$,
we obtain
\begin{equation}\label{def--Q}
\phi\dual_{\cW/\ccX}: 
\TT_{\cW/\ccX}\lra \EE_{\cW/\ccX}\defeq R\pi_{\cW\ast}\ev\sta\TT_{\cV/\cC}.
\end{equation}
By \cite[Prop.\,1.1]{CL}, it is a perfect obstruction theory.\footnote{This construction of $\phi\dual_{\cW/\ccX}$ applies to arbitrary representable $\cV\to\cC$. We restrict ourselves to bundle case for notational simplicity.}

\smallskip

In the following, we assume $\cV\to\cC$ is a (fixed) vector bundle. We consider two separate cases.
The first case we consider is when $\cV=\cV_1\oplus\cV_2$ is a direct sum of two vector bundles.
We continue to denote $\cW=C(\pi\lsta\cV)$. We introduce $\cW_i=C(\pi\lsta\cV_i)$.
Then the direct sum $\cV=\cV_1\oplus\cV_2$ induces a morphism
$\cW\lra \cW_1\times_\ccX\cW_2$,
which by direct check is an isomorphism.

There is another way to see this isomorphism. We let $\cC_{\cW_2}\defeq \cC\times_{\ccX}\cW_2$;
use (the same) $\pi: \cC_{\cW_2}\to\cW_2$ to denote its projection, and denote $\cV_{1,\cW_2}=\cV_1\times_\cC\cC_{\cW_2}$.


\begin{lemm}
We have canonical isomorphisms
$\cW\cong C(\pi\lsta(\cV_{1,\cW_2}))\cong \cW_1\times_\ccX\cW_2$.
\end{lemm}

Let $q_2:\cW\to\cW_2$ be the projection, as in the above lemma.
We let
$\phi\dual_{\cW/{\cW_2}}$, etc., 
be similarly defined perfect obstruction theories, as shown below,
\beq\label{sq-B}
\begin{CD}
\EE_{\cW/{\cW_2}}@>>>\EE_{\cW/\ccX}@>>>q_2\sta\EE_{{\cW_2}/\ccX}@>+1>>\\
@AA{\phi\dual_{\cW/{\cW_2}}}A@AA{\phi\dual_{\cW/\ccX}}A@AA{\phi\dual_{{\cW_2}/\ccX}}A\\
\TT_{\cW/{\cW_2}}@>>>\TT_{\cW/\ccX}@>>>q_2\sta\TT_{{\cW_2}/\ccX}@>+1>>,
\end{CD}
\eeq
where the top line is the d.t. induced by
$\cV=\cV_1\oplus\cV_2$, and the lower line is induced by $\cW\to\cW_2\to\ccX$.

\begin{prop}  \label{functorial}
The diagram \eqref{sq-B} is a morphism
between d.t.s.
\end{prop}

\begin{proof}  We form the diagram
\beq\label{1}
\begin{CD}
\cC_\cW@>\ev>>\cV@>\gamma_1>>\cV_1\\
@VV{\ti q_2}V@VV\gamma_2V@VVV\\
\cC_{{\cW_2}}@>{\ev_2}>> \cV_2@>>> \cC,
\end{CD}
\eeq
where $\gamma_i$ are projections induced by the direct sum $\cV=\cV_1\oplus\cV_2$;
and $\ti q_2$ is the lift of $q_2:\cW\to\cW_2$.
It induces a homomorphism between d.t.s
\beq\label{3}
\begin{CD}
\ev\sta\TT_{\cV/\cV_2}@>>> \ev\sta \TT_{\cV/\cC} @>>>\ti q_2\sta \ev_2\sta \TT_{\cV_2/\cC}@>+1>>\\
@AA{\psi}A@AAA@AAA\\
\TT_{\cC_\cW/\cC_{{\cW_2}}}@>>>\TT_{\cC_\cW/\cC}@>>>\ti q_2\sta \TT_{\cC_{{\cW_2}}/\cC}@>+1>>.
\end{CD}
\eeq
As $\cC\to\ccX$ is flat, the second row is equal to the pull back via $\pi_\cW:\cC_\cW\to\cW$ of the tangent complexes d.t. of the triple $\cW\to\cW_2\to\ccX$. Applying the projection formula to \eqref{3},
we obtain the following
morphism of d.t.s
$$\begin{CD}
R\pi_{\cW\ast}\ev\sta\TT_{\cV/\cV_2}@>>>R\pi_{\cW\ast} \ev\sta \TT_{\cV/\cC} @>>>R\pi_{\cW\ast}\ti q_2\sta \ev_2^{*} \TT_{\cV_2/\cC}@>+1>>\\
@AA{\psi_1}A@AA{\psi_2}A@AA{\psi_3}A\\
\TT_{\cW/{\cW_2}}@>>>\TT_{\cW/\ccX}@>>>q_2\sta \TT_{{\cW_2}/\ccX}@>+1>>\\
\end{CD}
$$
Note that by definition, $\EE_{\cW_2/\ccX}=R\pi_{\cW\ast}\ev_2\sta\TT_{\cV_2/\cC}$ and $\EE_{\cW/\cX}=
R\pi_{\cW\ast}\ev\sta\TT_{\cV/\cC}$. Because of the identity
$$R\pi_{\cW\ast}\ti q_2\sta \ev_2\sta \TT_{\cV_2/\cC}=
q_2\sta R\pi_{\cW\ast}\ev_2\sta\TT_{\cV_2/\cC},
$$
we see that $\psi_2=\phi\dual_{\cW/\ccX}$
and $\psi_3=q_2\sta\phi\dual_{\cW_2/\ccX}$.

It remains to show that $\psi_1=\phi\dual_{\cW/\cW_2}$.
Observe that $\psi_1$ is induced by the left square in \eqref{1},
and that square is identical to the left square in
\beq\label{2}
\begin{CD}
\cC_\cW@>\ev'>> \cV_{1,\cW_2} @>\pr>> \cV_1\\
@VVV@VVV@VVV\\
\cC_{\cW_2} @=\cC_{\cW_2}@>>>\cC.
\end{CD}
\eeq
Here $\ev'$ is the universal evaluation associated with the canonical $\cW \cong C(\pi\lsta(\cV_{1,\cW_2}))$.
Thus we have $\pr\circ\ev'=\gamma_1\circ \ev$, where $\gamma_1:\cV\to\cV_1$ is defined in \eqref{1}.

Since $\cV_1\to\cC$ is a bundle and thus is flat, we have $\TT_{\cV/\cV_2}\cong {\gamma_1\sta} \TT_{\cV_1/\cC}$;
thus the arrow $\psi_1$ equals
\beq\label{4}\TT_{\cC_\cW/\cC_{\cW_2}}\lra \ev\sta\gamma_1\sta\TT_{\cV_1/\cC}=
\pr\sta(\ev')\sta\TT_{\cV_1/\cC}= (\ev')\sta\TT_{\cV_{1,\cW_2}/\cC_{\cW_2}}.
\eeq
Here the last isomorphism is due to that $\pr\sta\TT_{\cV_1/\cC}\cong \TT_{\cV_{1,\cW_2}/\cC_{\cW_2}}$, as
$\cV_1\to \cC$ is smooth. On the other hand, it is evident that \eqref{4} is induced by $\ev'$. Therefore,
$$\EE_{\cW/\cW_2}\defeq R\pi_{\cW\ast}(\ev')\sta\TT_{\cV_{1,\cW_2}/\cC_{\cW_2}}=
R\pi_{\cW\ast}\ev\sta\TT_{\cV/\cV_2},
$$
and that $\psi_1=\phi\dual_{\cW/\cW_2}$. This proves the proposition.
 \end{proof}

\begin{rema} The natural diagram \eqref{sq-B} is commutative in case $\cV_1$ and $\cV_2$ are arbitrary Artin
stacks representable and quasi-projective over $\cC$, and $\cV_1\to \cC$ is flat. The proof is identical.
\end{rema}

The second case is when there is a (scheme) section of nodes $\cR\sub\cC$ that decomposes
$\cC$ into a union of two $\ccX$-families $\cC_1$ and $\cC_2$. We denote (the same) $\pi:\cC_i\to \ccX$ to be the projection. We let $\cV_i=\cV|_{\cC_i}$($=\cV\times_{\cC}\cC_i$), and define ${\cW_1}=C(\pi_{\ast}\cV_1)$.
We let
$$\phi\dual_{{\cW_1}/\ccX}: \TT_{{\cW_1}/\ccX}\lra \EE_{{\cW_1}/\ccX}$$
be the similarly defined perfect obstruction theory.
Note that for any $S$-family $(\sigma,s)$ in $\cW(S)$, letting
$\cC_{1,\sigma}=\cC_1\times_\cC\cC_\sigma$, 
then the family $(\sigma,s|_{\cC_{1,\sigma}})$ is a family in $\cW_1(S)$.
This defines a morphism
\beq\label{tau11}
\tau:\cW\lra {\cW_1}.
\eeq

To proceed, we like to rewrite $\tau$ along the line of a similar construction.
For $i=1$ and $2$, we let
$$\cC_{i,\cW_1}=\cC_i\times_{\ccX}\cW_1,\and \cV_{i,\cW_1}=\cV_i\times_{\cC_i}\cC_{i,\cW_1},
$$
with $\pi: \cC_{i,\cW_1}\to\cW_1$ its projection.

Let $\cS_1\in\Gamma(\cV_{1,\cW_1})$ be the universal section of $\cW_1$.
Let $\ti\cR=\cR\times_\cC\cC_{1,\cW_1}$ be the section (of $\cC_{1,\cW_1}\to\cW_1$) associated to
$\cR\sub\cC$. Then $\cS_1|_{\ti\cR}$ is a section of $\cV_{1,\cW_1}|_{\ti\cR}$.
Using $\cR=\cC_1\cap\cC_2$, we have $\cR\times_\cC\cC_{1,\cW_2}=\cR\times_\cC\cC_{2,\cW_1}$.
As $\cV_1$ and $\cV_2$ are respective restrictions of $\cV$,
$\cS_1|_{\ti\cR}$ is also a section of $\cV_{2,\cW_1}|_{\ti\cR}=\cV_{1,\cW_1}|_{\ti\cR}$.
We let $\Sigma\sub\cV_{2,\cW_1}$ be the substack
$\Sigma=\cS_1|_{\ti\cR}\sub \cV_{2,\cW_1}|_{\ti\cR}\sub\cV_{2,\cW_1}$.
We let $\text{Bl}_{\Sigma}(\cV_{2,\cW_1})$ be the blowing-up of $\cV_{2,\cW_1}$ along
$\Sigma$; we let
\beq\label{blow}
\cV_{2/1}= \text{Bl}_{\Sigma}(\cV_{2,\cW_1})-\{\text{the proper transform of }
\cV_{2,\cW_1}|_{\ti\cR}\sub \cV_{2,\cW_1}\}.
\eeq
We let $\pi: \cV_{2/1}\to\cW_1$ be the induced projection; we define
\beq\label{blcone}\cW_{2/1}=C(\pi\lsta\cV_{2/1}).
\eeq
 Note that $\cV_{2/1}$ is smooth over $\cC_{2,\cW_1}$.

We now construct a canonical (restriction) $\cW_1$-morphism
$\imath: \cW\to \cW_{2/1}$. Given any $\phi: S\to\cW$, associated to $(\sigma,s)\in\cW(S)$,
restricting $s$ to $\cC_1\times_{\ccX}S$ gives a family $(\sigma, s|_{\cC_1\times_{\ccX}S})\in\cW_1(S)$,
associating to the morphism $\tau(\phi):S\to\cW_1$.
The other part $s|_{\cC_2\times_{\ccX}S}$ is a section of the bundle
$$\cV_2\times_{\ccX}S=(\tau(\phi))\sta(\cV_{2,\cW_1})=\cV_{2,\cW_1}\times_{\tau(\phi),\cW_1}S.
$$
Because $s|_{\cC_1\times_{\ccX}S}$ and $s|_{\cC_2\times_{\ccX}S}$ are identical along
$\cR\times_{\ccX}S$, the section $s|_{\cC_2\times_{\ccX}S}$ lifts to a section of $(\tau(\phi))\sta(\cV_{2/1})$.
This defines a morphism
$\imath(\phi): S\lra \cW_{2/1}$,
commuting with $\phi: S\to\cW$, $\tau:\cW\to \cW_1$, and the
projection $\cW_{2/1}\to\cW_1$. As $\iota(\phi)$ is canonical, it defines a $\cW_1$-morphism
$\imath: \cW\lra \cW_{2/1}$.

\begin{lemm}
The morphism $\imath$ is an isomorphism. Let $\pr: \cW_{2/1}\to\cW_1$ be the tautological projection,
then $\pr\circ\imath=\tau$.
\end{lemm}

\begin{proof}
The proof follows directly from the construction.
\end{proof}

In the following, we will not distinguish $\cW$ and $\cW_{2/1}$ because of $\imath$.
Because all $\cW=\cW_{2/1}\to\cW_1$, $\cW\to\ccX$ and $\cW_1\to\ccX$ are
of the construction stated in the beginning of the Appendix, we have perfect obstruction theories $\phi\dual_{\bullet/\bullet}$ shown
\beq\label{sq-A}
 \begin{CD}
\EE_{\cW/{\cW_1}}@>\lam_1>>\EE_{\cW/\ccX}@>\lam_2>>\tau\sta\EE_{{\cW_1}/\ccX}@>+1>>\\
@AA{\phi\dual_{\cW/{\cW_1}}}A@AA{\phi\dual_{\cW/\ccX}}A@AA{\tau\sta\phi\dual_{{\cW_1}/\ccX}}A\\
\TT_{\cW/{\cW_1}}@>>>\TT_{\cW/\ccX}@>>>\tau\sta\TT_{{\cW_1}/\ccX}@>+1>>.
\end{CD}
\eeq
Here the lower sequence is the one induced by $\cW=\cW_{2/1}\to\cW_1\to\ccX$. The
arrow $\lam_1$ is induced by the canonical composite $\cV_{2/1}\to\cV$,
and $\lam_2$ is induced by the restriction of sheaves (bundles) $\cV\to\cV_2$.

\begin{prop}  \label{functorial0}
The the two rows in \eqref{sq-A} are d.t.s; the two squares in in \eqref{sq-A} are commutative.
Further, taking base change of \eqref{sq-A} via any $\xi\in\cW(\CC)$ and taking long exact sequences of
cohomology groups of the two rows, the vertical arrows induce a morphism between the two complexes of
vectos spaces.
\end{prop}

\begin{proof}
We denote by $\ev_1: \cC_{1,\cW}\to\cV_1$ and
$\ev:\cC_{\cW}\to\cV$ the obvious evaluation maps.
We have the following obvious fiber diagram
$$
\begin{CD}
\cC_{1,\cW}@>\ev_1>>\cV_1@>>>\cC_1\\
@VV{j}V@VVV@VVV\\
\cC_\cW@>\ev>>\cV@>>>\cC, \\
 \end{CD}
$$
where the vertical arrows are closed embeddings.
This implies that the square
\beq\label{aux1}
\begin{CD}
\TT_{\cC_{1,\cW}/\cC_1}@>d(\ev_1)>>\ev_1\sta\cV_1\\
@VV{u_1}V@VV{u_2}V\\
j\sta\TT_{\cC_\cW/\cC}@>j\sta d(\ev)>>j\sta \ev\sta\cV
\end{CD}
\eeq
is commutative.
Since  $\cC_{1,\cW}\sub \cC_\cW$ is fiber product of $\cC_1\sub\cC$ with $\cW\to\ccX$,
that $\cC$ and $\cC_1$ are flat over $\ccX$ implies that $\TT_{\cC_{1,\cW}/\cC_1}$
and $\TT_{\cC_\cW/\cC}$ are pullbacks of $\TT_{\cW/\ccX}$; thus $u_1$ is an isomorphism. Similarly,
$u_2$ is an isomorphism. This implies that the following square is commutative
$$
\begin{CD}
 \TT_{\cC_\cW/\cC}@>d(\ev)>>\ev\sta \TT_{\cV/\cC}= \ev\sta \cV\\
 @VV{u_1^{-1}\circ j\sta}V@VV{u_2^{-1}\circ j\sta }V\\
j\lsta \TT_{\cC_{1,\cW}/\cC_1}@>d(\ev_1)>>  j\lsta \ev_1\sta \TT_{\cV_1/\cC_1} = j\lsta \ev_1\sta\cV_1.
\end{CD}
$$

We let Let $\zeta:\cV_{2/1}\to \cV_{2,\cW_1}\to \cV_2\to\cV$ be the composite of the obvious morphisms.
Then we have the commutative square
$$\begin{CD}
 \cC_{2,\cW}\defeq \cC_2\times_{\ccX}\cW@>ev_{2/1}>>\cV_{2/1}\\
 @VVV@VV{\zeta}V\\
 \cC_\cW@>ev>>\cV.\\
 \end{CD}
$$
Here $\ev_{2/1}$ is defined using the universal section of $\cW_{2/1}$($=\cW$).

The above two squares induce the following two commutative squares of objects in $D^b(\sO_{\cC_\cW})$,
(letting $\pi_2:\cC_{2,\cW}\to\cW$ be the projection, letting
$\jmath_1: \cC_{1,\cW_1}\times_{\cW_1}\cW\to\cC_\cW$ and $\jmath_2: \cC_{2,\cW}\to\cC_\cW$ be the obvious
inclusions,)
\beq\label{auxdia}\begin{CD}
\jmath_{2\ast} ev_{2/1}\sta \TT_{\cV_{2/1}/\cC_{2,\cW_1}} @>>> ev\sta \TT_{\cV/\cC}@>>>
\jmath_{1\ast}  ev_1\sta \TT_{\cV_1/\cC_1} \\
  @AAA@AAA@AAA\\
\jmath_{2\ast} \pi_{2}\sta \TT_{\cW/\cW_1}=\TT_{\cC_{2,\cW}/\cC_{2,\cW_1}}@>>> \TT_{\cC_\cW/\cC} @>>> \jmath_{1\ast} \ti\tau\sta \TT_{\cC_{1,\cW_1}/\cC_1},
\end{CD}
\eeq
where $\ti\tau: \cC_{1,\cW}\to\cC_{1,\cW_1}$ is the projection lifting $\tau:\cW\to\cW_1$. (cf. \eqref{tau11}).


 Taking $\pi:\cC_\cW\to\cW$ 
and $\pi_1:\cC_{1,\cW_1}\to\cW_1$ to be the
respective projections, let $\bar {ev}_1:\cC_{1,\cW_1}\to\cV_1$
be the evaluation using the universal section of $\cW_1$,
applying $R\pi_{\ast}$ to \eqref{auxdia}, we
obtain commutative diagrams
 \beq\label{auxdia1}\begin{CD}
  R\pi_{2\ast} ev_{2/1}\sta \TT_{\cV_{2/1}/\cC_{2,\cW_1}} @>>> R\pi_{\ast}ev\sta \TT_{\cV/\cC}@>>>
\tau\sta  R\pi_{1\ast} \bar{ev}_1\sta \TT_{\cV_1/\cC_1} \\
  @AAA@AAA@AAA\\
  R\pi_{2\ast}  \pi_{2}\sta \TT_{\cW/\cW_1}
   @>>>R\pi_{\ast} \TT_{\cC_\cW/\cC} @>>>\tau\sta R\pi_{1\ast} \TT_{\cC_{1,\cW_1}/\cC_1}\\
   @AAA@AAA@AAA\\
  \TT_{\cW/{\cW_1}}@>>>\TT_{\cW/\ccX}@>>>\tau\sta\TT_{{\cW_1}/\ccX},
   \end{CD}
\eeq
Note that the first row is identical to the first row of \eqref{sq-A}, and the composited three vertical arrows in
\eqref{auxdia1} are $\phi\dual_{\cW/\cW_1}$, $\phi\dual_{\cW/\ccX}$ and $\tau\sta\phi\dual_{\cW_1/\ccX}$ in
\eqref{sq-A}.

On the other hand, we have canonical $ev_{2/1}\sta\TT_{\cV_{2/1}/\cC_{2,\cW_1}}\cong
ev\sta (\cV|_{\cC_2}(-\cR))$ (due to the blowing-up construction), the first row of \eqref{auxdia} equals to
\beq\label{exact1}
0\lra   ev\sta (\cV|_{\cC_2}(-\cR)) \lra ev\sta \cV\lra ev\sta (\cV|_{\cC_1})\lra 0,
\eeq
thus is a distinguished triangle. Therefore, the first row of \eqref{auxdia1}, which is the first row of \eqref{sq-A},
is a distinguished triangle.
Finally, the further part of the proposition is implied by commutativity of the following diagram
$$
 \begin{CD}
h^0(\EE_{{\cW_1}/\ccX}|_{\tau(\xi)}) @>>>h^1(\EE_{\cW/{\cW_1}}|_\xi)  \\
@AA{h^1(\phi\dual_{{\cW_1}/\ccX}|_\xi)}A@AA{h^1(\phi\dual_{\cW/{\cW_1}}|_\xi)}A\\
h^0(\TT_{{\cW_1}/\ccX}|_{\tau(\xi)}) @>>>h^1(\TT_{\cW/{\cW_1}}|_\xi) \end{CD}
$$
which can be checked by $\check{\text{C}}$ech cohomology description of the obstruction class assignment.
We leave it to the reader.
 \end{proof}

\end{document}